\documentclass[a4paper,11pt]{article}

\usepackage{fullpage}
\usepackage{amsbsy}
\usepackage{latexsym}
\usepackage{amsfonts}
\usepackage{amssymb}
\usepackage[usenames]{color}
\usepackage{amsmath,amsthm}
\usepackage{enumerate}

\usepackage{graphicx}

\usepackage[T1]{fontenc}

\usepackage{mathdots}

\def\RR{\rm \hbox{I\kern-.2em\hbox{R}}}
\def\NN{\rm \hbox{I\kern-.2em\hbox{N}}}
\def\ZZ{\rm {{\rm Z}\kern-.28em{\rm Z}}}
\def\CC{\rm \hbox{C\kern -.5em {\raise .32ex \hbox{$\scriptscriptstyle
|$}}\kern
-.22em{\raise .6ex \hbox{$\scriptscriptstyle |$}}\kern .4em}}
\def\vp{\varphi}
\def\<{\langle}
\def\>{\rangle}
\def\t{\tilde}

\def\e{\varepsilon}

\newcommand{\bX}{\mathbf{X}}
\newcommand{\bG}{\mathbf{G}}
\newcommand{\bM}{\mathbf{M}}
\newcommand{\bI}{\mathbf{I}}
\newcommand{\bc}{\mathbf{c}}
\newcommand{\bd}{\mathbf{d}}

\def\R{\mathbb{R}}
\def\N{\mathbb{N}}

\def\E{\mathbb{E}}

\def\Chi{\raise .3ex
\hbox{\large $\chi$}} \def\vp{\varphi}
\def\lsima{\hbox{\kern -.6em\raisebox{-1ex}{$~\stackrel{\textstyle<}{\sim}~$}}\kern -.4em}
\def\lsim{\hbox{\kern -.2em\raisebox{-1ex}{$~\stackrel{\textstyle<}{\sim}~$}}\kern -.2em}
\def\gsim{\hbox{\kern -.2em\raisebox{-1ex}{$~\stackrel{\textstyle>}{\sim}~$}}\kern -.2em}
\def\({\Bigl (}
\def\){\Bigr )}

\newcommand{\be}{\begin{equation}}
\newcommand{\ee}{\end{equation}}
\newcommand{\bea}{$$ \begin{array}{lll}}
\newcommand{\eea}{\end{array} $$}
\newcommand{\bi}{\begin{itemize}}
\newcommand{\ei}{\end{itemize}}
\newcommand{\iref}[1]{(\ref{#1})}

\newtheorem{theorem}{Theorem}[section]
\newtheorem{lemma}[theorem]{Lemma}

\theoremstyle{definition}

\newtheorem{remark}[theorem]{Remark}

\numberwithin{equation}{section}

\usepackage{algorithm}
\usepackage{algpseudocode}
\algrenewcommand\algorithmicrequire{\textbf{input:}}
\algrenewcommand\algorithmicensure{\textbf{output:}}

\providecommand{\abs}[1]{\lvert#1\rvert}

\providecommand{\norm}[1]{\lVert#1\rVert}

\providecommand{\ceil}[1]{\lceil#1\rceil}

\newcommand{\EE}{\mathbb{E}}

\DeclareMathOperator{\Prob}{Pr}

\title{\Large{\textbf{Sequential sampling for optimal weighted least\\[2pt]  squares approximations in hierarchical spaces}}\thanks{Benjamin Arras is supported by the European Research Council under grant ERC AdG 338977 BREAD.
Markus Bachmayr acknowledges support by the Hausdorff Center of Mathematics, University of Bonn.
Albert Cohen is supported by the Institut Universitaire de France and
by the European Research Council under grant ERC AdG 338977 BREAD.
}}
\author{Benjamin Arras\thanks{Sorbonne Universit\'es, UPMC Univ Paris 06, CNRS, UMR 7598, Laboratoire Jacques-Louis Lions, 4 place Jussieu, 75005 Paris, France (arras@ljll.math.upmc.fr)}, \ Markus Bachmayr\thanks{Hausdorff Center for Mathematics \& Institute for Numerical Simulation, Wegelerstr.\ 6, 53115 Bonn, Germany (bachmayr@ins.uni-bonn.de)} \ and Albert Cohen\thanks{Sorbonne Universit\'es, UPMC Univ Paris 06, CNRS, UMR 7598, Laboratoire Jacques-Louis Lions, 4 place Jussieu, 75005 Paris, France (cohen@ljll.math.upmc.fr)}}
\date{\today}

\begin{document}
\maketitle
\begin{abstract}
We consider the problem of approximating an unknown function $u\in L^2(D,\rho)$
from its evaluations at given sampling points $x^1,\dots,x^n\in D$, where $D\subset \R^d$ is a general
domain and $\rho$ a probability measure. The approximation is picked in a
linear space $V_m$ where $m=\dim(V_m)$ and computed by a weighted least squares method. Recent
results show the advantages of picking the sampling points at random
according to a well-chosen probability measure $\mu$ that depends both on $V_m$ and $\rho$.
With such a random design, the weighted least squares approximation is proved to be 
stable with high probability, and having precision comparable to that of the exact 
$L^2(D,\rho)$-orthonormal projection onto $V_m$,
in a near-linear sampling regime $n\sim{m\log m}$. The present paper is motivated
by the adaptive approximation context, in which one typically generates 
a nested sequence of spaces $(V_m)_{m\geq1}$ with increasing dimension. 
Although the measure $\mu=\mu_m$ changes with $V_m$, it is possible to recycle
the previously generated samples by interpreting $\mu_m$ as a mixture
between $\mu_{m-1}$ and an update measure $\sigma_m$. Based on this observation,
we discuss sequential sampling algorithms that maintain the stability
and approximation properties uniformly over all spaces $V_m$. Our main result
is that the total number of computed sample at step $m$ remains of the order $m\log{m}$
with high probability. Numerical experiments confirm this analysis.

\end{abstract}

\section{Introduction}

Least squares approximations are ubiquitously used in numerical computation
when trying to reconstruct an unknown function $u$ defined on some domain $D\subseteq \R^d$
from its observations $y^1,\dots,y^n$ at 
a limited amount of points $x^1,\dots,x^n\in D$. In its simplest form the method amounts to minimizing
the least squares fit
\be
\frac 1 n\sum_{i=1}^n |y^i-v(x^i)|^2,
\ee
over a set of functions $v$ that are subject to certain constraints expressing a prior on the
unknown function $u$. There are two classical approaches for imposing such constraints:
\begin{itemize}
\item[(i)] Add a penalty term $P(v)$ to the least squares fit. Classical instances include
norms of reproducing kernel Hilbert spaces or $\ell^1$ norms that promote sparsity of $v$
when expressed in a certain basis of functions. 
\item[(ii)] Limit the search of $v$ to a space $V_m$ of finite dimension $m\leq n$. Classical instances
include spaces of algebraic or trigonometric polynomials, wavelets, or splines.
\end{itemize}
The present paper is concerned with the second approach, in which approximability 
by the space $V_m$ may be viewed as a prior on the unknown function.
We measure accuracy in the Hilbertian norm
\be
\|v\|=\biggl(\int_D |v(x)|^2 d\rho\biggr)^{1/2}=\|v\|_{L^2(D,\rho)},
\ee
where $\rho$ is a probability measure over $D$. We denote by $\<\cdot,\cdot\>$ the associated inner product.
The error of best approximation is defined by
\be
e_m(u):=\min_{v\in V_m} \|u-v\|,
\ee
and is attained by $P_mu$, the $L^2(D,\rho)$-orthogonal projection of $u$ onto $V_m$. Since the least squares
approximation $\t u$ is picked in $V_m$, it is natural to compare $\|u-\t u\|$ with $e_m(u)$. In particular, the method
is said to be near-optimal (or instance optimal with constant $C$) if the comparison
\be
\|u-\t u\|\leq Ce_m(u),
\ee 
holds for all $u$, where $C>1$ is some fixed constant. 

The present paper is motivated by applications where the sampling points $x^i$ are not prescribed 
and can be chosen by the user. Such a situation typically occurs when
evaluation of $u$ is performed either by a computer simulation or a physical experiment,
depending on a vector of input parameters $x$ that can be set by the user. This evaluation is typically
costly and the goal is to obtain a satisfactory {\it surrogate model} $\t u$ from a minimal number of evaluations
\be
y^i=u(x^i).
\ee
For a given probability measure $\rho$ and approximation space $V_m$ of interest, a relevant question
is therefore whether instance optimality can be achieved with sample size $n$ that is moderate, ideally linear in $m$. 

Recent results of \cite{Do1,JNZ} for polynomial spaces and \cite{CM1} in 
a general approximation setting show that this objective can be achieved 
by certain random sampling schemes in the
more general framework of {\em weighted least squares} methods. The approximation $\t u$
is then defined as the solution to
\be
\min_{v\in V_m} \frac 1 n\sum_{i=1}^n w(x_i)  |y^i-v(x^i)|^2,
\label{wls}
\ee
where $w$ is a positive function and the $x^i$ are independently drawn according to 
a probability measure $\mu$, that satisfy the constraint
\be
w \,d\mu= d\rho.
\label{cons}
\ee
The choice of a sampling measure $\mu$ that differs from the error norm measure $\rho$
appears to be critical in order to obtain instance optimal approximations 
with an optimal sampling budget.

In particular, it is shown in \cite{Do1,CM1}, that
there exists an optimal choice of $(\rho,\mu)$ such that the weighted least squares
is stable with high probability and instance optimal in expectation, under 
the near-linear regime $n\sim m$ up to logarithmic factors. The optimal sampling measure and weights
are given by
\be
d\mu_m=\frac {k_m}{m} d\rho \quad {\rm and}\quad w_m=\frac m{k_m},
\label{optimeasure}
\ee
where $k_m$ is the so-called Christoffel function defined by
\be
k_m(x)=\sum_{j=1}^m |\vp_j(x)|^2,
\ee
with $\{\vp_1,\dots,\vp_m\}$ any $L^2(D,\rho)$-orthonormal basis of $V_m$. 

In many practical applications, the space $V_m$ is picked within a 
family $(V_m)_{m\geq 1}$ that has the nestedness property
\be\label{nestedspaces}
V_1\subset V_2\subset \cdots
\ee
and accuracy is improved by raising the dimension $m$. The sequence  $(V_m)_{m\geq 1}$ may
either be a priori defined, or adaptively generated, which means
that the way $V_m$ is refined into $V_{m+1}$ may depend on the result of the least squares computation.
Example of such hierarchical adaptive or non-adaptive schemes include in particular:
\begin{itemize} 
\item[(i)] Mesh refinement in low-dimension performed by progressive addition 
of hierarchical basis functions or wavelets, which is relevant for approximating piecewise smooth functions, 
such as images or shock profiles, see \cite{Co,Da,De}.
\item[(ii)] Sparse polynomial approximation in high dimension, which is relevant for 
the treatment of certain parametric and stochastic PDEs, see \cite{CD}.
\end{itemize}
In this setting, we are facing the difficulty that the optimal measure $\mu_m$
defined by \iref{optimeasure} varies together with $m$. 

In order to maintain an optimal sampling budget,
one should avoid the option of drawing a new sample 
\be
S_m=\{x^1_m,\dots,x^n_m\}
\ee
of increasing
size $n=n(m)$ at each step $m$. In the particular case where $V_m$ are the univariate polynomials of degree $m-1$,
and $\rho$ a Jacobi type measure on $[-1,1]$, it is known that $\mu_m$ converges weakly to the equilibrium measure
defined by
\be
d\mu^*(y)=\frac {dy}{\pi\sqrt{1-y^2}},
\ee
with a uniform equivalence
\be
c_1\mu^* \leq \mu_m\leq c_2\mu^*,
\label{equivmeas}
\ee
see \cite{EMN,MT}. This suggests the option of replacing all $\mu_m$ by the single $\mu^*$, as studied in
\cite{JNZ}. Unfortunately,
such an asymptotic behaviour is not encountered for most general choices of spaces $(V_m)_{m\geq 1}$. An equivalence
of the form \iref{equivmeas} was proved in \cite{HNTW} for
sparse multivariate polynomials, however with a ratio $c_2/c_1$ that increases exponentially with the dimension,
which theoretically impacts in a similar manner the sampling budget needed for stability.

In this paper, we discuss sampling strategies that are based on the observation that the optimal measure $\mu_m$ 
enjoys the mixture property
\be
\mu_{m+1}=\(1-\frac 1 {m+1}\) \mu_{m}+\frac 1 {m+1} \sigma_{m+1}, \quad {\rm where} \quad d\sigma_m:=|\vp_m|^2 d\rho.
\label{mixt}
\ee
As noticed in \cite{Do3}, this leads naturally to sequential sampling strategies, where the sample $S_m$
is recycled for generating $S_{m+1}$. The main contribution of this paper is to analyze such a sampling strategy
and prove that the two following properties can be jointly achieved in an expectation or high probability sense:
\begin{enumerate}
\item
Stability and instance optimality of weighted least squares hold uniformly over all $m\geq 1$.
\item
The total sampling budget after $m$ step is linear in $m$ up to logarithmic factors.
\end{enumerate}

The rest of the paper is organized as follows. We recall in \S 2 stability and approximation estimates
from \cite{Do1,CM1} concerning the weighted least squares method in a fixed space $V_m$. 
We then describe in \S 3 a sequential sampling strategy based on \iref{mixt} and establish the 
above optimality properties 1) and 2). These optimality properties are numerically illustrated
in \S 4 for the algorithm and two of its variants. 

\section{Optimal weighted least squares}

We denote by $\|\cdot\|_n$ the discrete Euclidean norm defined by
\be
\|v\|_n^2:=\frac 1 n \sum_{i=1}^n w(x^i) |v(x^i)|^2,
\ee
and by $\<\cdot,\cdot\>_n$ the associated inner product.
The solution $\t u\in V_m$ to \iref{wls} may be thought of as an orthogonal projection
of $u$ onto $V_m$ for this norm. Expanding
\be
\t u=\sum_{j=1}^m c_j\vp_j,
\ee
in the basis $\{\vp_1,\dots,\vp_m\}$ of $V_m$, the coefficient vector $\bc=(c_1,\dots,c_m)^T$ is solution to
the linear system
\be
\bG_m \bc = \bd,
\ee
where $\bG_m$ is the Gramian matrix for the inner product $\<\cdot,\cdot\>_n$ with entries
\be\label{gramiandef}
\bG_{j,k}:=\<\vp_j,\vp_k\>_n=\frac 1 n \sum_{i=1}^n w(x^i)\vp_j(x^i)\vp_k(x^i),
\ee
and the vector $\bd$ has entries $\bd_k=\frac 1 n\sum_{i=1}^n w(x^i)y^i\vp_k(x^i)$. The solution $\bc$ always exists
and is unique when $\bG_m$ is invertible. 

Since $x^1,\dots,x^n$ are drawn independently according to $d\mu$, the relation \iref{cons} implies
that $\E(\<v_1,v_2\>_n)=\<v_1,v_2\>$, and in particular 
\be
\E(\bG_m)=\bI.
\ee
The stability and accuracy analysis of the weighted least squares method can be related to 
the amount of deviation between $\bG_m$ and its expectation $\bI$ measured in the spectral norm. Recall that for $m\times m$ matrices $\bM$, this norm is defined as $\|\bM\|_2=\sup_{\|v\|_2=1} \|\bM v\|_2$.
This deviation also describe the closeness of the norms $\|\cdot\|$ and $\|\cdot\|_n$ over the space $V_m$, since
one has
\be
\|\bG_m-\bI\|_2\leq \delta\quad  \iff  \quad  (1-\delta)\|v\|^2 \leq \|v\|_n^2\leq (1+\delta)\|v\|^2,\quad v\in V_m.
\ee
Note that this closeness also implies a bound 
\be
\kappa(\bG_m)\leq \frac {1+\delta}{1-\delta},
\ee
on the condition number of $\bG_m$. 

Following \cite{CM1}, we use the particular value $\delta=\frac 1 2$,
and define $\t u$ as the solution to \iref{wls} when $\|\bG_m-\bI\|_2\leq \frac 1 2$ and set $\t u=0$ otherwise.
The probability of the latter event can be estimated by a matrix tail bound, noting that
\be
\bG_m=\frac 1 n \sum_{i=1}^n \bX^i,
\ee 
where the $\bX^i$ are $n$ independent realizations of the rank one
random matrix 
\be
\bX:=w(x)(\vp_j(x)\vp_k(x))_{j,k=1,\dots,m},
\ee
where $x$ is distributed according to $\mu$. The matrix Chernoff bound, see Theorem 1.1 in \cite{Tr}, gives
\be
\Prob \( \|\bG_m-\bI\|_2\geq \frac 1 2\)\leq 2m\exp(-\gamma n/K), \quad \gamma:=\frac 1 2(1-\ln 2).
\label{cher}
\ee
where $K:=\|w k_m\|_{L^\infty}$
is an almost sure bound for $\|\bX\|$. Since $\int wk_md\mu=\int k_m d\rho=m$,
it follows that $K$ is always larger than $m$.
With the choice $\mu=\mu_m$ given by \iref{optimeasure} for the 
sampling measure, one has exactly $K=m$, which leads to the following result.

\begin{theorem}\label{thm:ncond}
Assume that the sampling measure and weight function are given by \iref{optimeasure}. Then
for any $0<\e<1$, the condition
\be\label{ncondition}
n\geq  cm (\ln(2m)-\ln(\e)), \quad c:=\gamma^{-1}=\frac 2{1-\ln 2},
\ee
implies the following stability and instance optimality properties:
\be 
\Prob \( \|\bG_m-\bI\|_2\geq \frac 1 2\)\leq \e,\quad m\geq 1.
\label{prob}
\ee
\end{theorem}
\noindent
The probabilistic inequality \iref{prob} induces an instance optimality estimate in expectation, since
in the event where $\|\bG_m-\bI\|_2\leq \frac 1 2$, one has
\be
\begin{aligned}
\|u-\t u\|^2 &=\|u-P_mu\|^2+\|\t u-P_m u\|^2 \\
& \leq e_m(u)^2+ 2\|\t u-P_m u\|_n^2 \leq e_m(u)^2+2\|u-P_mu\|_n^2,
\end{aligned}
\ee
and therefore
\be
\E(\|u-\t u\|^2)\leq e_m(u)^2+2\E(\|u-P_mu\|_n^2)+\e \|u\|^2=3 e_m(u)^2 +\e \|u\|^2.
\label{exp}
\ee
By a more refined reasoning, the constant $3$ can be replaced by $1+\frac {c}{\ln(2m)-\ln(\e)}$ which
tends to $1$ as $m/\e$ becomes large, see \cite{CM1}.

In summary, when using the optimal sampling measure $\mu_m$ defined
by \iref{optimeasure}, stability and instance optimality can be achieved in the 
near linear regime
\be
\label{nmdef}
n= n_\varepsilon(m) := \ceil{ c\, m\, (\ln (2m) - \ln \e )},
\ee 
where $\varepsilon\!\in\,]0,1[$ controls the probability of failure. To simplify notation later, we set $n_\e(0) := 0$.

\section{An optimal sequential sampling procedure}\label{sec:seq}

In the following analysis of a sequential sampling scheme, we assume a sequence of nested spaces $V_1 \subset V_2 \subset \cdots$ and corresponding basis functions $\vp_1,\vp_2,\ldots$ to be given such that for each $m$,
$\{\vp_1,\dots,\vp_m\}$ is an orthonormal basis of $V_m$. In practical applications,
such spaces may either be fixed in advance, or adaptively selected, that is,
 the choice of $\vp_{m+1}$ depends on the computation of the weighted least squares approximation \eqref{wls} for $V_m$.
 In view of the previous result, one natural objective is to genererate sequences of samples $(S_m)_{m\geq 1}$ 
 distributed according to the different measures $(\mu_m)_{m\geq 1}$, with
 \be
 \#(S_m)=n_\e(m),
 \ee
 for some prescribed $\e>0$.

The simplest option for generating such sequences would be directly drawing samples from $\mu_m$ for each $m=1,2,\dots$ separately. Since we ask that 
$n_\e(m)$ is proportional to $m$ up to logarithmic factors, this leads to a total cost $C_m$
after $m$ step given by
\be
C_m=\sum_{k=1}^m \#(S_k),
\ee
which increases faster than quadratically with $m$.
Instead, we want to recycle the existing samples $S_m$ in generating $S_{m+1}$ to arrive at a scheme
such that the total cost $C_m$ remains comparable to $n_\e(m)$, that is, close to linear in $m$.

To this end, we use the mixture property \iref{mixt}.
In what follows, we assume a procedure for sampling from each update measure $d\sigma_{j} := \abs{\vp_j}^2 d\rho$ to be available. In the univariate case, standard methods are inversion transform sampling or rejection sampling.
These methods may in turn serve in the multivariate case when the $\vp_j$ are tensor product basis functions on a product domain 
and $\rho$ is itself of tensor product type, since $\sigma_j$ are then product measures that can be sampled via their univariate factor measures.
We first observe that, in order to draw $x$ distributed according to $\mu_m$, for some fixed $m$, we can proceed as follows:
\be\label{mixturedirect}~
\text{Draw $j$ uniformly distributed in $\{1,\ldots, m\}$, then draw $x$ from $\sigma_j$.}
\ee
Let now $(n(m))_{m\geq 1}$ be an increasing sequence representing the prescribed size of the samples $(S_m)_{m\geq 1}$.
Suppose that we are given $S_m=\{x^1_m ,\ldots, x^{n(m)}_m\}$ i.i.d. according to $\mu_{m}$.
In order to obtain the new sample $S_{m+1}=\{x^1_{m+1} ,\ldots, x^{n(m+1)}_{m+1}\}$  i.i.d. according to  $\mu_{m+1}$, we can proceed as stated in
the following Algorithm \ref{mixturestep}.

\begin{algorithm}~
\caption{Sequential sampling}\label{mixturestep}
\begin{algorithmic}
\vspace{-6pt}
\Require sample $S_m = \{ x_m^1,\ldots, x_m^{n(m)}\}$ from $\mu_m$
\Ensure sample $S_{m+1} = \{ x_{m+1}^1,\ldots, x_{m+1}^{n(m+1)}\}$ from $\mu_{m+1}$
\vspace{3pt}
\hrule
\vspace{6pt}

\For{$i=1,\ldots,n(m)$}
\State draw $a_i$ uniformly distributed in $\{ 1, \ldots, m+1 \}$
\If{$a_i = m+1$}
\State draw $x^i_{m+1}$ from $\sigma_{m+1}$
\Else
\State set $x^i_{m+1} := x^i_m$
\EndIf
\EndFor

\For{$i=n(m)+1, \ldots, n(m+1)$}
\State draw $x^i_{m+1}$ from $\mu_{m+1}$ by \eqref{mixturedirect}
\EndFor
\end{algorithmic}
\end{algorithm}

Algorithm \ref{mixturestep} requires a fixed number $n(m+1) - n(m)$ of samples from $\mu_{m+1}$ and an additional number $\t n(m)$ of samples from $\sigma_{m+1}$. The latter is a random variable that can be expressed as 
\be
\t n(m) = \sum_{i=1}^{n(m)} b^i_{m+1},
\ee
where for each fixed $m\geq 1$, the $(b^i_m)_{i=1,\ldots, n(m)}$ are i.i.d.\ Bernoulli random variables with $\Prob(b^i_m = 1) = \frac1{m}$. Moreover, $\{b^i_m \,\colon\: i=1,\ldots, n(m), \,m\geq 1\}$ is a collection of independent random variables. This immediately gives an expression for the total cost $C_m$ after $m$ successive applications of Algorithm \ref{mixturestep}, beginning with $n(1)$ samples from $\mu_1$,
as the random variable
\be\label{Cmdef}
 C_m := n(m) + s(m), \quad s(m) := \sum_{k=1}^{m-1} \t n(k) = \sum_{k=1}^{m-1} \sum_{i=1}^{n(k)} b^i_{k+1} .
\ee
We now focus on the particular choice
\be
\label{defstdn}
n(m):=n_\e(m),
\ee
as in \iref{nmdef}, for a prescribed $\e\in ]0,1[$. This particular choice ensures that, for all $m\geq 1$,
\be
\label{thm1prob}
\Prob \bigl( \norm{\bG_m - \bI}_2 \geq  \textstyle\frac12\displaystyle \bigr) \leq \varepsilon,
\ee
where $\bG_m$ denotes the Gramian for $V_m$ according to \eqref{gramiandef}.

We first estimate $\EE(s(m))$ for this choice of $n(m)$. For this purpose, we note that $\EE(\t n(k)) = \frac{n(k)}{k+1}$
and  use the following lemma.

\begin{lemma}\label{Smest}
For $m \geq 1$ and $\e>0$,
\be\label{eq:expectationbounds}
 \textstyle\frac12\displaystyle n(m) - 2c   \leq  \sum_{k=1}^{m} \frac{ n(k) }{k+1} \leq  n(m)  +  1 .
\ee
\end{lemma}

\begin{proof}
For the upper bound, we note that $n(k) \leq  c k ( \ln (2k)  - {\ln \e }) + 1$ and
\be\label{upper1}\begin{aligned} 
  \sum_{k=1}^{m} \frac1{k+1} \bigl( c k (\ln (2k) - {\ln\e})  + 1 \bigr)
    &\leq c \sum_{k=1}^{m} \bigl(  \ln(2k) - {\ln \e} \bigr) + \sum_{k=1}^{m} \frac1{k+1} \\
     & \leq c \bigl( m \ln 2  +  \ln m! - m {\ln\e}\bigr) + \ln (m+1),
     \end{aligned}
\ee
By the Stirling bound $k ! \leq e k^{k+\frac12} e^{-k}$ for $k \geq 1$,
\be\label{stirling}
   \ln m! \leq m \bigl( \ln m  - 1 \bigr) +\textstyle \frac12 \displaystyle \ln m  + 1.
\ee
Combining this with \eqref{upper1} gives
\be\label{upper2}
\begin{aligned}
  \sum_{k=1}^{m} \frac{ n(k) }{k+1} &\leq c  m \bigl(  \ln 2 m  - {\ln \e} \bigr) - c (m-1) + \frac{c}2 \ln  m +\ln (m+1) \\
    & \leq n(m) + 1,
  \end{aligned}
\ee
where the inequality $- c (m-1) + \frac{c}2 \ln  m +\ln (m+1) \leq 1$ for $m\geq 1$ is verified for the choice of $c$ in \eqref{ncondition} by direct evaluation for $m=1$ and monotonicity. For the lower bound, we estimate
\be
   \sum_{k=1}^{m} \frac{ n(k) }{k+1} \geq   \frac{c m}{2} (  \ln 2 - {\ln\e})  +  c \sum_{k=1}^{m} \frac{k}{k+1} \ln k .
\ee
Using monotonicity and integration by parts with $\frac{d}{dx}\left(x-\ln(1+x)\right)=\frac{x}{x+1}$, we obtain
\be\label{sumintlowerbound}
\begin{aligned}
  \sum_{k=1}^{m} \frac{k}{k+1} \ln k & \geq \int_1^{m} \frac{x}{x+1} \ln x \,dx \\
   &= - \int_1^{m} \frac{x-\ln(x+1)}{x}\,dx + \bigl[ (x - \ln(x+1)) \ln x \bigr]_1^{m}  \\
   & = - (m- 1) + \int_1^{m} \frac{\ln (x+1)}{x}\,dx  + (m-\ln (m+1)) \ln m .   \\
\end{aligned}\ee
Moreover,
\be
   \int_1^{m} \frac{\ln (x+1)}{x}\,dx = \int_0^{\ln m} \ln (1 + e^t) \,dt \geq \int_0^{\ln m} t\,dt = \frac12 \ln^2m ,
\ee
and using this in \eqref{sumintlowerbound} gives
\be\label{lower1}
  \sum_{k=1}^{m} \frac{k}{k+1} \ln k \geq \frac12 m \ln  m + \frac1{2c} + R(m) ,
\ee
where
\be\label{lower2}
  R(m) =  \frac12 m (\ln m - 2)  - \ln (m+1) \ln m + \frac12 \ln^2 m  + 1  - \frac1{2c} > -2.
\ee
The latter inequality can be directly verified for the first few values of $m$ and then follows for $m\geq 1$ by monotonicity.
Thus, we obtain
\be\label{lower3}
   \sum_{k=1}^{m} \frac{ n(k) }{k+1} \geq  \textstyle \frac{1}2\displaystyle \bigl[ c m \bigl( \ln 2m  - {\ln \e}\bigr)  + 1 \bigr] -2c \geq \textstyle\frac12\displaystyle n(m) - 2c,
\ee
which concludes the proof of the lemma.
\end{proof}
\noindent
As an immediate consequence, we obtain an estimate of the total cost in expectation by
\begin{theorem}\label{thm:mean}
With $n(m)=n_\e(m)$, the total cost after $m$ steps of Algorithm 1 satisfies
\be\label{costexpectationbound}
 n(m) + \textstyle \frac 1 2 \displaystyle n(m-1) - 2c\leq  \EE(C_m) \leq n(m) + n(m-1) + 1.
\ee
\end{theorem}

We next derive estimates for the probability that the upper bound in the above estimate
is exceeded substantially by $C_m$.
The following Chernoff inequality can be found in \cite{AV}. For convenience of the reader, we give a standard short proof following \cite{HR}.

\begin{lemma}\label{chernoffbound}
Let $N\geq 1$, $(p_i)_{1\leq i\leq N}\in [0,1]^N$ and $(X_i)_{1\leq i\leq N}$ be a collection of independent Bernoulli random variables with $\Prob(X_i = 1) = p_i$ for all $1\leq i\leq N$. Set $\bar X = \EE (X_1+\ldots + X_N) = p_1+\ldots+p_N$.
	Then, for all $\tau \geq 0$,
	\[
	 \Prob \bigl( X_1+\ldots+X_N \geq ( 1+ \tau) \bar X \bigr) \leq ( 1+ \tau)^{-(1+\tau)\bar X} e^{\tau \bar X}.
	\]
	In particular, for all $\tau \in [0,1]$,
	\[
	\Prob\bigl( X_1+\ldots+X_N \geq ( 1+ \tau) \bar X \bigr) \leq e^{-\tau^2 \bar X/3}.
	\]
\end{lemma}

\begin{proof}
	For $t\geq 0$ and $Y := X_1+\ldots+X_N$, $ \Prob(Y \geq ( 1 + \tau) \bar X)  \leq e^{-t (1+\tau) \bar X} \EE (e^{t Y})$ by Markov's inequality,
	and using that $1 + a \leq e^a$, $a \in \R$,
\[
		\EE (e^{t Y}) = \prod_{i=1}^N \EE(e^{tX_i}) = \prod_{i=1}^N \bigl( p_i e^t + (1-p_i)\bigr) 
		 \leq \prod_{i=1}^N e^{p_i(e^t-1)} = e^{(e^t-1)\bar X}. 
\]
Now take $t=\ln(1+\tau)$. Moreover, for $\tau \in [0,1]$, one has $\tau - (1+\tau) \ln(1+\tau) \leq - \textstyle\frac13\displaystyle \tau^2$.
\end{proof}
\noindent
We now apply this result to the random part $s(m)$ of the total sampling costs $C_m$ as defined in \eqref{Cmdef},
and obtain the following probabilistic estimate.

\begin{theorem}\label{thm1}
With $n(m)=n_\e(m)$, the total cost after $m$ step of Algorithm 1 satisfies, for any $\tau \in [0,1]$,
 \be
 \begin{aligned}
   \Prob \bigl( C_m \geq n(m) + (1 + \tau) (n(m-1)  + 1) \bigr)  &  \leq  M_\tau  e^{- \frac{\tau^2}{6} n(m-1) }  \\
      & \leq  M_\tau    \biggl(\frac {2(m-1)}{\e}\biggr)^{ - \frac{\tau^2 c}6  (m - 1)}
  \end{aligned}
 \ee
 with $M_\tau := e^{\frac{2 c \tau^2}{3}}$.
\end{theorem}

\begin{proof}
We apply Proposition \ref{chernoffbound} to the independent variable $b^i_k$ which appear in $s(m)$,
which gives 
\be
\Prob\bigl(C_m\geq n(m)+ (1+\tau) \E(s(m)) \bigr)=\Prob\bigl(s(m)\geq (1+\tau) \E(s(m)) \bigr)\leq e^{-\tau^2 \E(s(m))/3}.
\ee
The result follows by using the lower bound on $\EE(s(m))$ in Lemma \ref{Smest}.
\end{proof}
\noindent

With the choice $n(m)=n_\e(m)$ we are ensured that, for each value of $m$ separately,
 the failure probability is bounded by $\varepsilon$. When the intermediate results in each step are used to drive an adaptive selection of the sequence of basis functions, however, it will typically be of interest to ensure the stronger uniform statement that with high probability, $\norm{\bG_m - \bI}_2 \leq  \frac12$ for \emph{all} $m$.
To achieve this, we now consider a slight modification of the above results with $m$-dependent choice of failure probability to ensure that $\bG_m$ remains well-conditioned with high probability jointly for all $m$. We define the sequence of failure probabilities
\be\label{emdef}
 \e(m) = \frac{6 \e_0}{(\pi m)^2}, \quad m \geq 1,
\ee
for a fixed $\e_0 \in ]0,1[$, and now analyze the repeated application of Algorithm \ref{mixturestep}, now using 
\be\label{defNm}
n(m):=n_{\e(m)}(m)=\ceil{ c\, m\, (\ln (2m) + 2\ln m - \ln (6\e_0/\pi^2)  ) }
\ee
samples for $V_m$. Note that $n(m)$ differs only by a further term of order $\log m$ from $n_{\e_0}(m)$.

\begin{lemma}
 For $m\geq 1$, let $n(m)$ be defined as in \eqref{defNm} with $\e(m)$ as in \eqref{emdef}. Then
 \be\label{eq:expectationbounds2}
   \textstyle\frac12\displaystyle n(m) - 6c   \leq  \sum_{k=1}^{m} \frac{ n(k) }{k+1} \leq  n(m)  +  1 .
 \ee
\end{lemma}

\begin{proof}
For the upper bound in \eqref{eq:expectationbounds2}, note that $n(k) \leq c k \bigr( \ln (2k) +  \ln k^2 - {\ln (6\e_0/\pi^2)} \bigr) + 1$. Using \eqref{stirling},
\be\label{bound:lnk2}
  \sum_{k=1}^m \ln k^2   \leq  m \ln m^2 + 2\bigl( 1 - m + \textstyle \frac12\displaystyle \ln m \bigr) \\
    \leq m \ln m^2.
\ee
Thus, for all $m\geq 1$,
\begin{align*}
\sum_{k=1}^m \dfrac{n(k)}{k+1}&\leq c\sum_{k=1}^m \dfrac{k}{k+1}\left(\ln(2k)-\ln\left(\frac{6\varepsilon_0}{\pi^2}\right)\right)+\sum_{k=1}^m\frac{1}{k+1}+c\sum_{k=1}^{m}\frac{k\ln k^2}{1+k}\\
&\leq c  m \bigl(  \ln 2 m  - {\ln (6\e_0/\pi^2)} \bigr)+1+ cm\ln m^2\\
&\leq n(m)  +  1,
\end{align*}
where we have used  \eqref{upper1}, \eqref{upper2} together with \eqref{bound:lnk2}. Let us deal with the lower bound. For all $m\geq 1$,
\begin{align}\label{lbnek}
\sum_{k=1}^m \dfrac{n(m)}{k+1}&\geq \sum_{k=1}^m\dfrac{1}{k+1}\left(ck \left(\ln (2k)-\ln \left( \frac{6\varepsilon_0}{\pi^2} \right)+\ln k^2\right)\right)\nonumber\\
&\geq \frac{1}{2}\left(cm\left(\ln(2m)-\ln\left(\frac{6\varepsilon_0}{\pi^2}\right) \right)+1\right)-2c+\sum_{k=1}^mc\frac{k}{k+1}\ln k^2.
\end{align}
Moreover using \eqref{lower1} and \eqref{lower2}
\[
  c \sum_{k=1}^m \frac{k}{k+1}\ln k^2 \geq 2c \left(\frac{m}{2}\ln m+\frac{1}{2c}-2\right).
\]
Thus, combining the previous bound together with \eqref{lbnek} and the definition of $n(m)$ concludes the proof of the lemma.
\end{proof}
\noindent
As a consequence, analogously to \eqref{costexpectationbound} we have 
\be\label{mean2}
  \EE(C_m) \leq n(m) + n(m-1) + 1.
\ee
Using the above lemma as in Theorem \ref{thm1} combined with a union bound, we arrive at the following
uniform stability result.

\begin{theorem}\label{thm2}
Let $\e(m)$ be defined as in \eqref{emdef}. Then applying Algorithm 1 with $n(m)$ as in \eqref{defNm}, one has
\[ 
    \Prob \bigl( \exists m\in \N \colon  \norm{\bG_m - \bI}_2 \geq \textstyle\frac12\displaystyle \bigr) \leq \varepsilon_0 , 
\]
and for any $\tau \in [0,1]$ and all $m\geq 1$, the random variable $C_m$ satisfies
 \be\begin{aligned}
   \Prob \bigl( C_m \geq n(m) + (1 + \tau) (n(m-1)  + 1) \bigr)   & \leq  M_\tau  e^{- \frac{\tau^2}{6} n(m-1) }  \\
      & \leq  M_\tau   \biggl(\frac {\pi^2  (m-1)}{3\e_0}\biggr)^{ - \frac{\tau^2 c}2  (m - 1)} 
       \end{aligned}
 \ee
 with $M_\tau := e^{2 c \tau^2}$.
\end{theorem}
\noindent
In summary, applying Algorithm \ref{mixturestep} successively to generate the samples $S_1, S_2,\ldots$, we can ensure that $\norm{\bG_m - \bI}_2 \leq \frac12$ holds uniformly for all steps with probability at least $1 - \e_0$. The corresponding total costs for generating $S_1,\ldots, S_m$ can exceed a fixed multiple of $n(m)$ only with a probability that rapidly approaches zero as $m$ increases.

\begin{remark}
Algorithm \ref{mixturestep} and the above analysis can be adapted to create samples from $\mu_{m+q}$, $q\geq 2$, using those for $\mu_m$, which corresponds to adding $q$ basis functions in one step.
In this case,
\be
  \mu_{m+q} = \frac{m}{m+q}\mu_m + \frac{q}{m+q} \sum_{j=m+1}^{m+q} \frac1q \sigma_{j},
\ee
where samples from $\sum_{j=m+1}^{m+q} \frac1q \sigma_{j}$ can be obtained by mixture sampling as in \eqref{mixturedirect}.
\end{remark}

\section{Numerical illustration}

In our numerical tests, we consider two different types of orthonormal bases and corresponding target measures $\rho$ on $D\subseteq \R$:
\begin{enumerate}[(i)]
\item On the one hand, we consider the case where $D$ is equal to $\mathbb{R}$, $\rho$ is the standard Gaussian measure on $\mathbb{R}$ and $V_m$ the vector space spanned by the Hermite polynomials normalized to one with respect to the norm $\|.\|$ up to degree $m-1$, for all $m\geq 1$. This case is
an instance of polynomial approximation method where the  function $u$ and its approximants 
from $V_m$ might be unbounded in $L^\infty$, as well as the
Christoffel function $k_m(x):=\sum_{j=1}^{m}\left|H_{j-1}(x)\right|^2$.

\item On the other hand, we consider the case where $D=[0,1]$ with $\rho$ the uniform measure on $[0,1]$ and the approximation spaces $(V_m)_{m\geq 1}$ are generated by Haar wavelets refinement. The Haar wavelets are
of the form $\psi_{l,k}=2^{l/2}\psi(2^l-k)$ with $l\geq 0$ and $k=0,..., 2^l-1$, and $\psi:=\Chi_{[0,1/2[}-\Chi_{[1/2,1[}$. 
In adaptive approximation, the spaces
$V_m$ are typically generated by including as the scale level $l$ grows, the values of $k$ such that the coefficient of $\psi_{l,k}$
for the approximated function is expected to be large. This can be described by growing a finite tree 
within the hierarchical structure induced by the dyadic indexing of the Haar wavelet family: the indices $(l+1,2k)$ or $(l+1,2k+1)$
can be selected only if $(l,k)$ has already been. In our experiment, we 
generate the spaces $V_m$ by letting such a tree grow at random, starting from $V_1 = {\rm span}\{ \psi_{0,0}\}$. This 
selection of $(V_m)_{m\geq 1}$ is done once and used for all further tests.
The resulting sampling measures $(\mu_m)_{m\geq 1}$ exhibit the local refinement behavior 
of the corresponding approximation spaces $(V_m)_{m\geq 1}$.
\end{enumerate}

As seen further, although the spaces and measures are quite different, the cost of 
the sampling algorithm behaves similarly in these two cases. While we have used univariate domains $D$
for the sake of numerical simplicity, we expect a similar behaviour in multivariate cases, since
our results are also immune to the spatial dimension $d$. As already mentioned, the sampling method
is then facilitated when the functions $\vp_j$ are tensor product functions and $\rho$ is
a product measure.

\begin{figure}
\begin{tabular}{cc}
	\includegraphics[width=7.5cm]{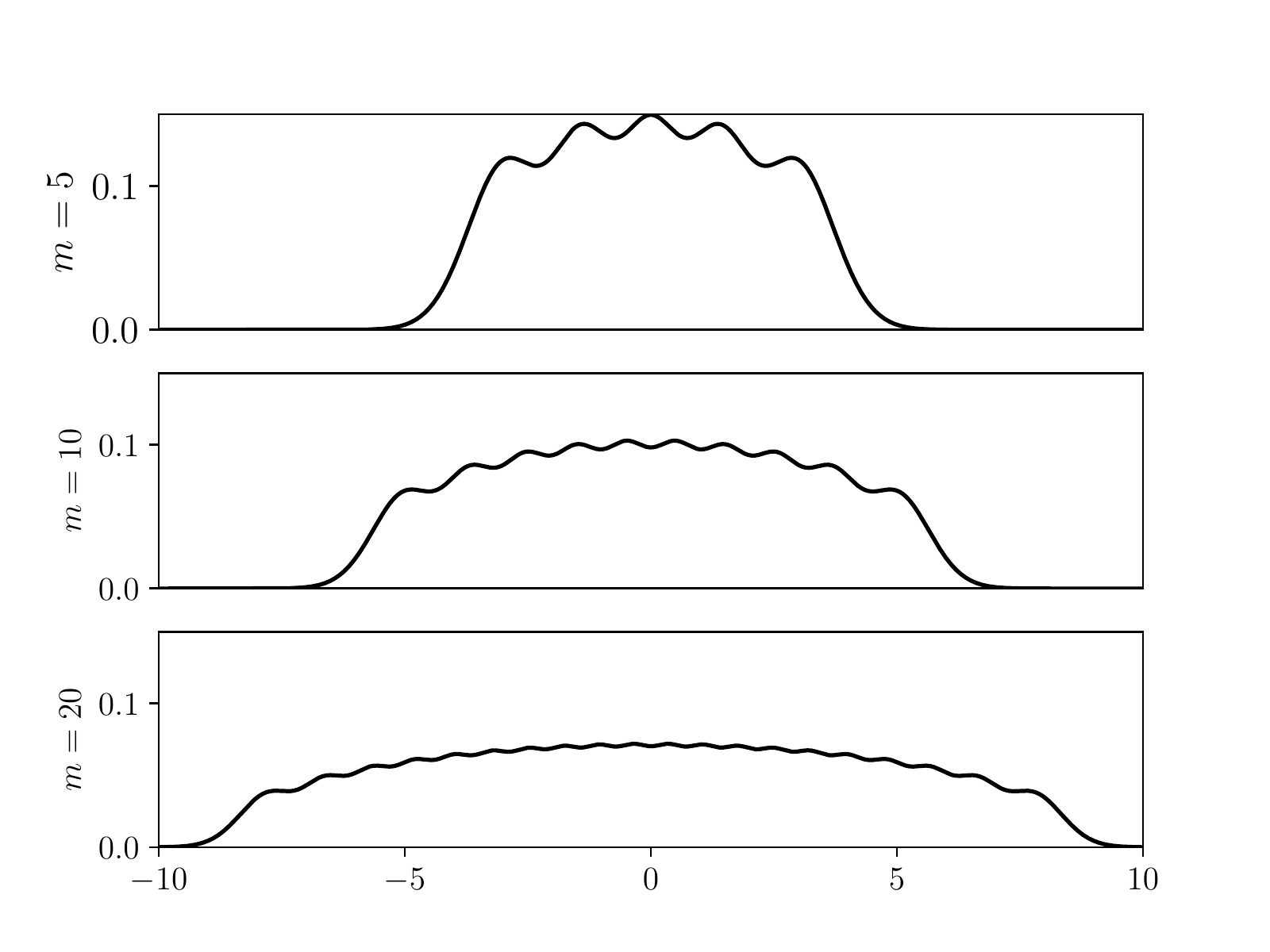} & \includegraphics[width=7.5cm]{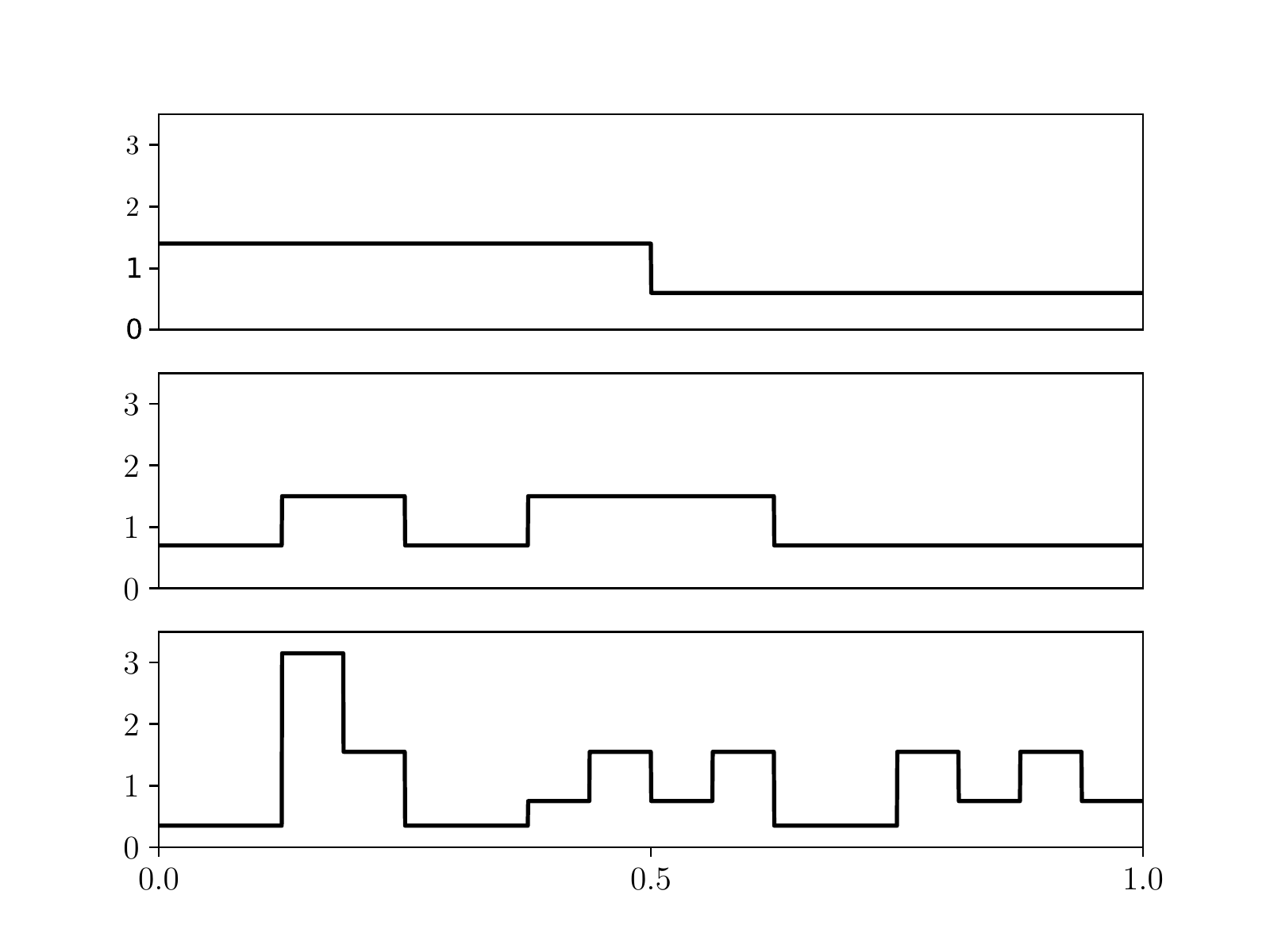} \\
	\small (a)  Hermite polynomials &\small  (b) Haar wavelets
	\end{tabular}
	\caption{Sampling densities $\mu_m$ for (a) Hermite polynomials of degrees $0,\ldots,m-1$, (b) subsets of Haar wavelet basis selected by random tree refinement.}\label{figsmp}
\end{figure}

The sampling measures $\mu_m$ are shown in Figure \ref{figsmp}. In  case (ii), the measures $\sigma_j$ are uniform measures on dyadic intervals in $]0,1[$, from which we can sample directly, using $\mathcal{O}(1)$ operations per sample. In case (i), we have instead $d\sigma_j = \abs{H_{j-1}}^2 d\rho$. Several strategies for sampling from these densities are discussed in \cite{CM1}. For our following tests, we use inverse transform sampling: we draw $z$ uniformly distributed in $[0,1]$, and obtain a sample $x$ from $\sigma_j$ by solving $\Phi_j(x):=\int_{-\infty}^x \abs{H_{j-1}}^2 d\rho = z$. To solve these equations, in order to ensure robustness, we use a simple bisection scheme. Each point value of $\Phi_j$ can be obtained using $\mathcal{O}(j)$ operations, exploiting the fact that for $j\geq 2$,
\[
  \Phi_j = \Phi-\sum_{k = 1}^{j-1} \frac{H_k H_{k-1}}{\sqrt{k}} g,
\]
where $g$ is the standard Gaussian density and $\Phi$ its cumulative distribution function.
The bisection thus takes $\mathcal{O}(j \log \epsilon)$ operations to converge to accuracy $\epsilon$. For practical purposes, the sampling for case (i) can be accelerated by precomputation of interpolants for $\sigma_j$.

In each of our numerical tests, we consider the distributions of $\kappa(\bG_m)$ and $C_m$ resulting from the sequential generation of  sample sets $S_m$ for $m=1,\ldots,50$ by Algorithm \ref{mixturestep}. In each case, the quantiles of the corresponding distributions are estimated from 1000 test runs. Figure \ref{fig1nherm} shows the results for Algorithm \ref{mixturestep} in case (i) with $n(m) = n_\e(m)$ as in \eqref{defstdn}, where we choose $\e = 10^{-2}$. We know from Theorem \ref{thm:ncond} that, for each $m$, one has $\kappa(\bG_m)\leq 3$ with probability greater than $1-\e$. In fact, this bound appears to be fairly pessimistic, since no sample with $\kappa(\bG_m)> 3$ is encountered in the present test. In Figure \ref{fig1nherm}(b), we show the ratio $C_m / n_\e(m)$. Recall that $C_m$ is defined in \ref{Cmdef} as the total number of samples used in the repeated application of Algorithm \ref{mixturestep} to produce $S_1,\ldots,S_m$, and thus $C_m / n_\e(m)$ provides a comparison to the costs of directly sampling $S_m$ from $\mu_m$. The results are in agreement with Theorems \ref{thm:mean} and \ref{thm1}, which show in particular that $C_m < 2 n_\e(m)$ with very high probability.
\begin{figure}
\begin{tabular}{cc}
	\includegraphics[width=8cm]{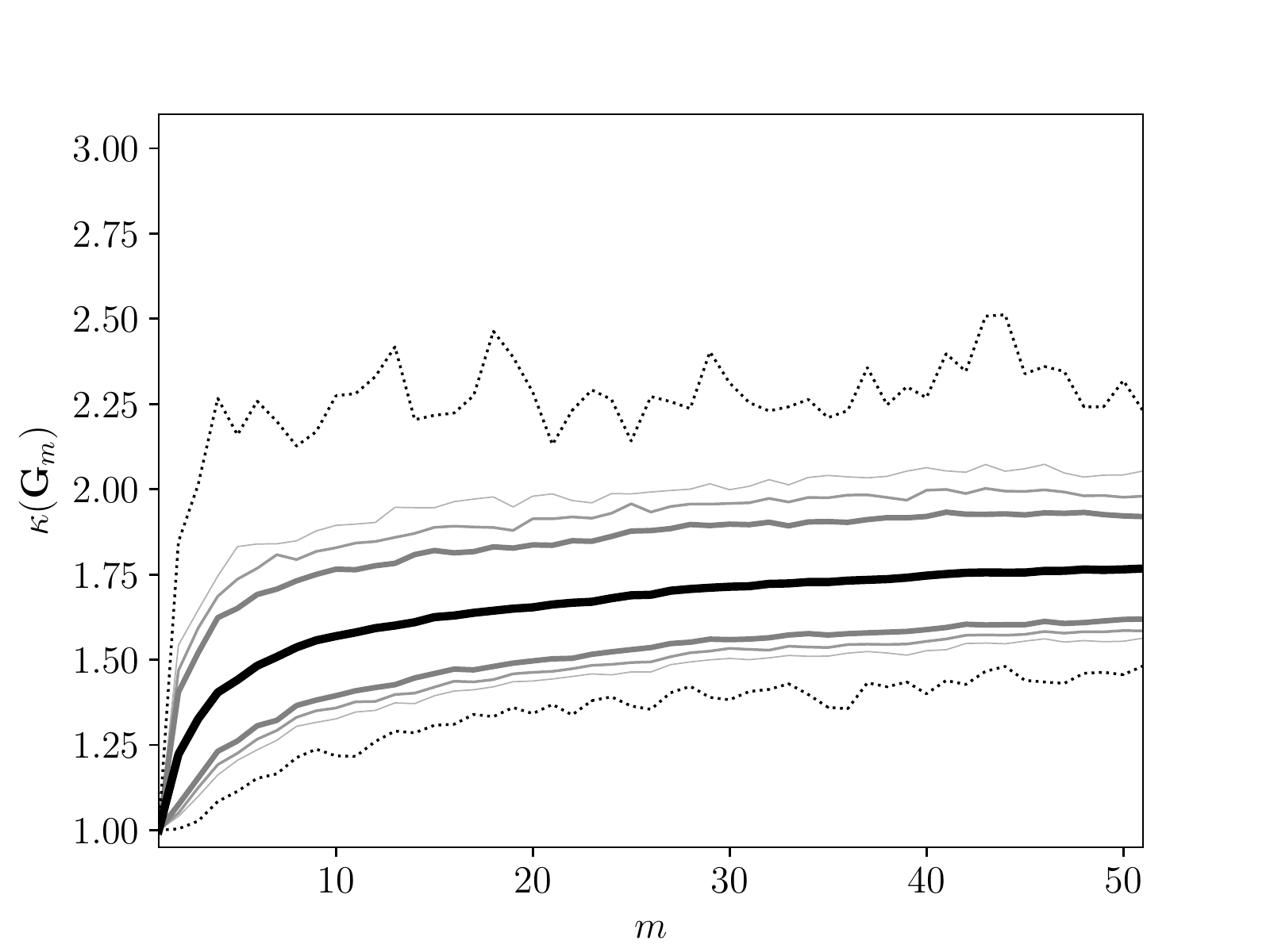} & \includegraphics[width=8cm]{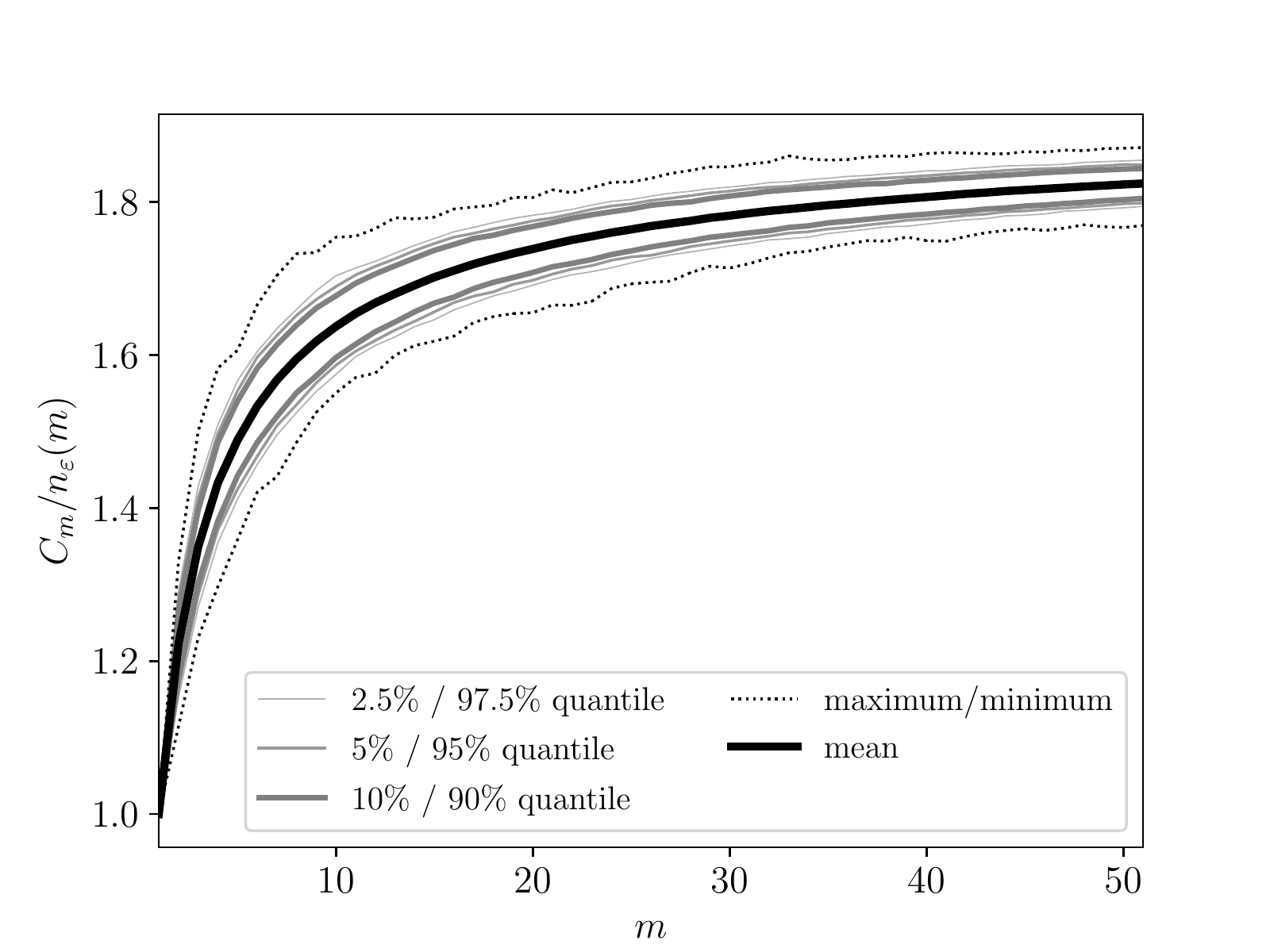} \\
	\small (a)  Gramian condition numbers & \small (b) Cost comparison
	\end{tabular}
	\caption{Results for Algorithm \ref{mixturestep} with $n(m) = n_\e(m)$ as in \eqref{defstdn}, $\e = 10^{-2}$, applied to Hermite polynomials of degrees $0,\ldots,m-1$.}\label{fig1nherm}
\end{figure}

Using the same setup with $n(m) = n_{\e(m)}(m)$ as in \eqref{defNm}, where $\e_0 = 10^{-2}$, leads to the expected improved uniformity in $\kappa(\bG_m)$. The corresponding results are shown in Figure \ref{fig1neherm}, with sampling costs that are in agreement with \eqref{mean2} and Theorem \ref{thm2}. Since the effects of replacing $n_{\e}(m)$ by $n_{\e(m)}(m)$ are very similar in all further tests, we only show results for $n(m) = n_{\e}(m)$ with $\e = 10^{-2}$ in what follows.
\begin{figure}
\begin{tabular}{cc}
	\includegraphics[width=7.5cm]{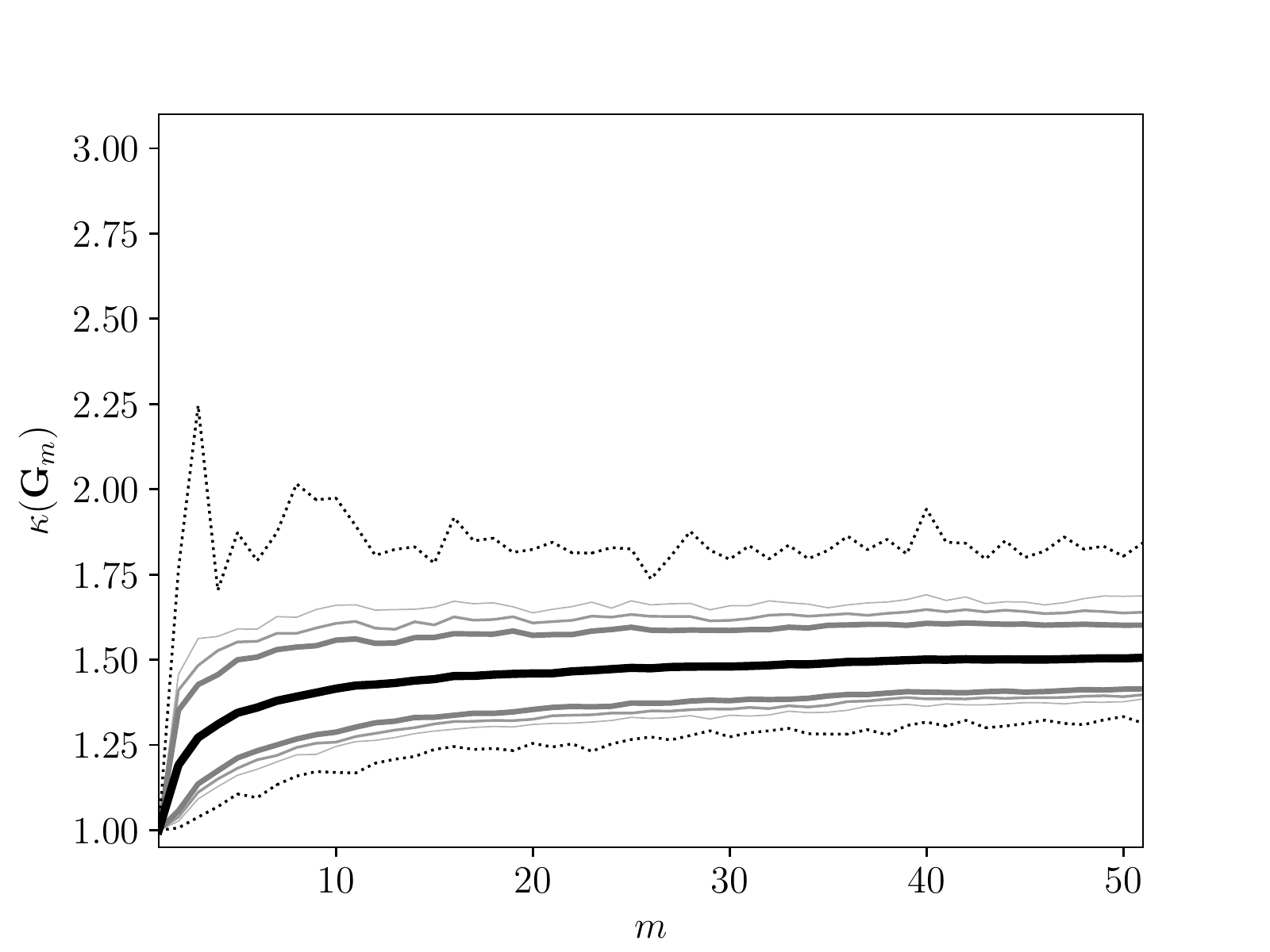} & \includegraphics[width=7.5cm]{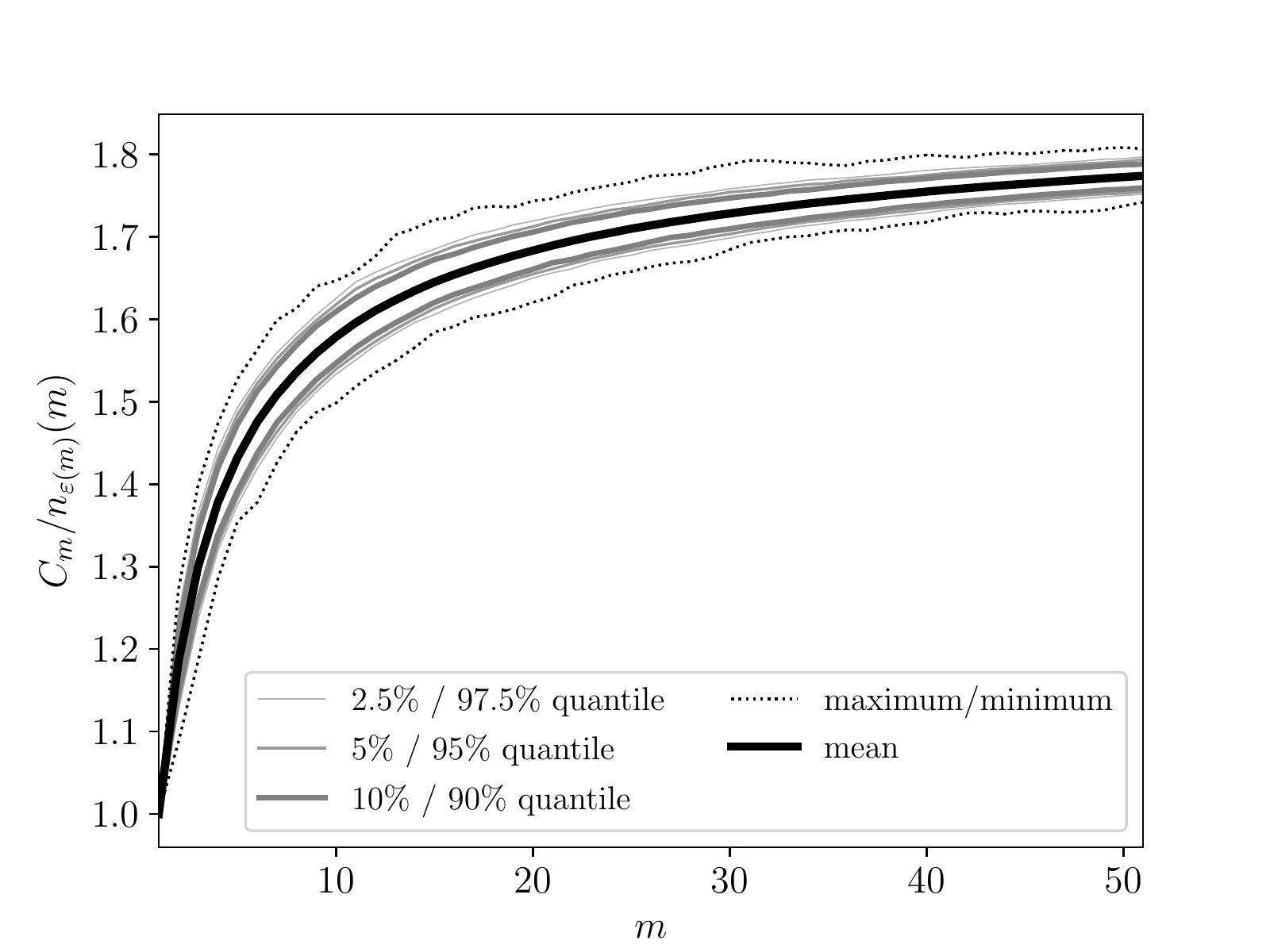}\\
	\small (a)  Gramian condition numbers & \small (b) Cost comparison
	\end{tabular}
	\caption{Results for Algorithm \ref{mixturestep} with $n(m) = n_{\e(m)}(m)$ as in \eqref{defNm}, $\e_0 = 10^{-2}$, applied to Hermite polynomials of degrees $0,\ldots,m-1$.}\label{fig1neherm}
\end{figure}

While the simple scheme in Algorithm \ref{mixturestep} already ensures near-optimal costs with high probability, there are some practical variants that can yield better quantitative performance.
A first such variant is given in Algorithm \ref{mixturestep2}. Instead of deciding for each previous sample separately whether it will be re-used, here a queue of previous samples is kept, from which these are extracted in order until the previous sample set $S_m$ is exhausted. Clearly, the costs of this scheme are bounded from above by those of Algorithm \ref{mixturestep}.
\begin{algorithm}~
\caption{Sequential sampling with sample queue}\label{mixturestep2}
\begin{algorithmic}
\vspace{-6pt}
\Require sample $S_m = \{ x_m^1,\ldots, x_m^{n(m)}\}$ from $\mu_m$
\Ensure sample $S_{m+1} = \{ x_{m+1}^1,\ldots, x_{m+1}^{n(m+1)}\}$ from $\mu_{m+1}$
\vspace{3pt}
\hrule
\vspace{6pt}
\State $j:= 1$
\For{$i=1,\ldots,n(m+1)$}
\State draw $a_i$ uniformly distributed in $\{ 1, \ldots, m+1 \}$
\If{$a_i = m+1$}
\State draw $x^i_{m+1}$ from $\sigma_{m+1}$
\ElsIf{$j \leq n(m)$}
\State  $x^i_{m+1} := x^j_m$
\State $j \gets j+1$ 
\Else
\State draw $x^i_{m+1}$ from $\mu_{m}$ by \eqref{mixturedirect}
\EndIf
\EndFor
\end{algorithmic}
\end{algorithm}

As expected, in the results for Algorithm \ref{mixturestep2} applied in case (i), which are shown in Figure \ref{fig2nherm}, we find an estimate of the distribution of $\kappa(\bG_m)$ that is essentially identical to the one for Algorithm \ref{mixturestep} in Figure \ref{fig1nherm}(a). The costs, however, are substantially more favorable than the ones in Figure \ref{fig1nherm}(b): using Algorithm \ref{mixturestep2}, the successive sampling of $S_1,\ldots, S_m$ uses only a small fraction of additional samples when compared to directly sampling only $S_m$.
\begin{figure}
\begin{tabular}{cc}
	\includegraphics[width=7.5cm]{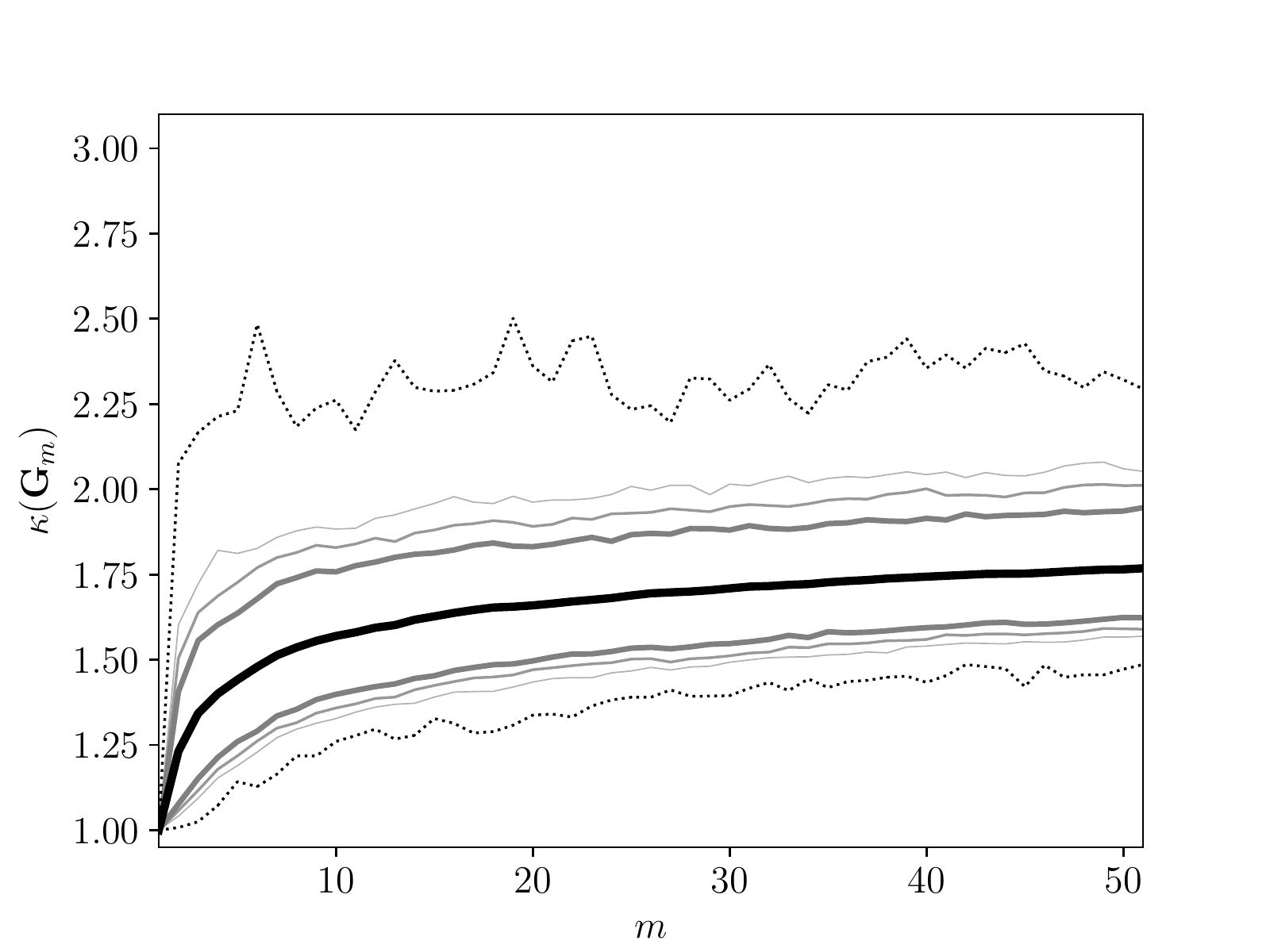} & \includegraphics[width=7.5cm]{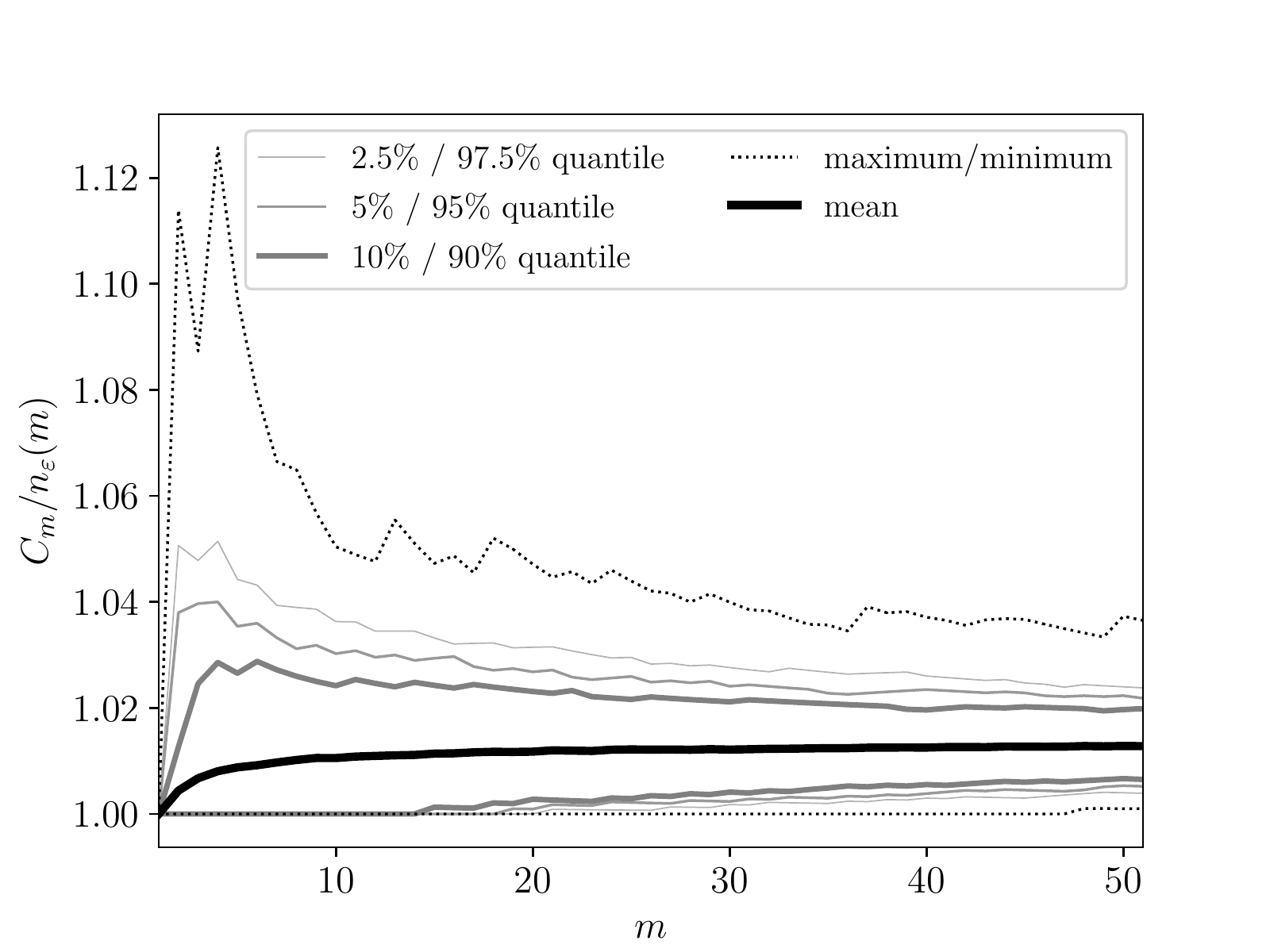}\\
	\small (a)  Gramian condition numbers & \small (b) Cost comparison
	\end{tabular}
	\caption{Results for Algorithm \ref{mixturestep2} with $n(m) = n_\e(m)$ as in \eqref{defstdn}, $\e = 10^{-2}$, applied to Hermite polynomials of degrees $0,\ldots,m-1$.}\label{fig2nherm}
\end{figure}

Figures \ref{fig1nhaar} and \ref{fig2nhaar} show the analogous comparison of Algorithms \ref{mixturestep} and \ref{mixturestep2} applied to case (ii), which leads to very similar results. This is not surprising, considering that the bounds in the general Theorem \ref{thm:ncond} on optimal least squares sampling, as well as those in \S\ref{sec:seq}, are all independent of the chosen $L^2$-space and of the corresponding orthonormal basis. Our numerical results thus indicate that one can indeed expect a rather minor effect of this choice also in practice.
\begin{figure}
\begin{tabular}{cc}
	\includegraphics[width=7.5cm]{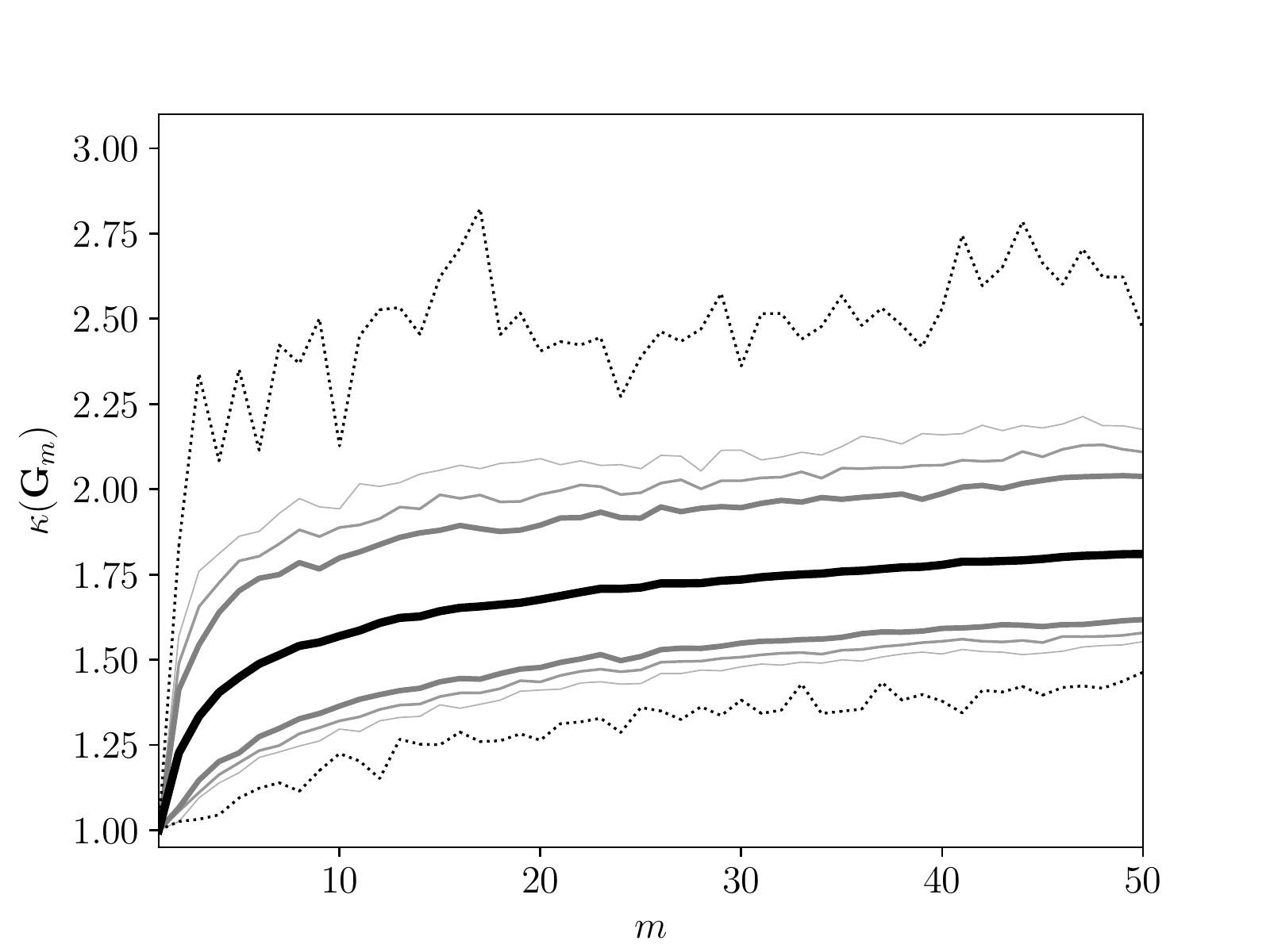} & \includegraphics[width=7.5cm]{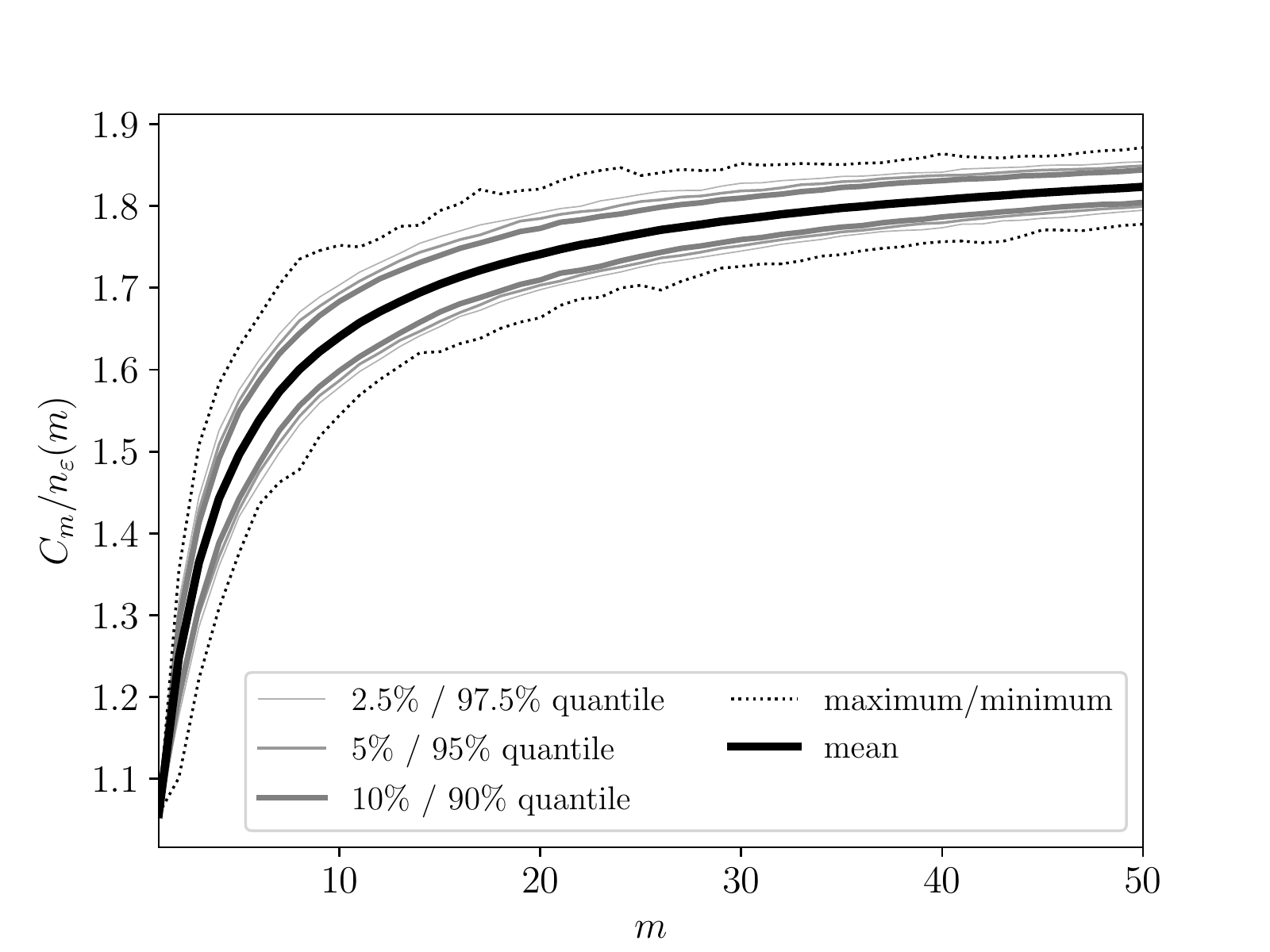}\\
	\small (a)  Gramian condition numbers & \small (b) Cost comparison
	\end{tabular}
	\caption{Results for Algorithm \ref{mixturestep} with $n(m) = n_\e(m)$ as in \eqref{defstdn}, $\e = 10^{-2}$, applied to subset of Haar basis obtained by random tree refinement.}\label{fig1nhaar}
\end{figure}

\begin{figure}
\begin{tabular}{cc}
	\includegraphics[width=7.5cm]{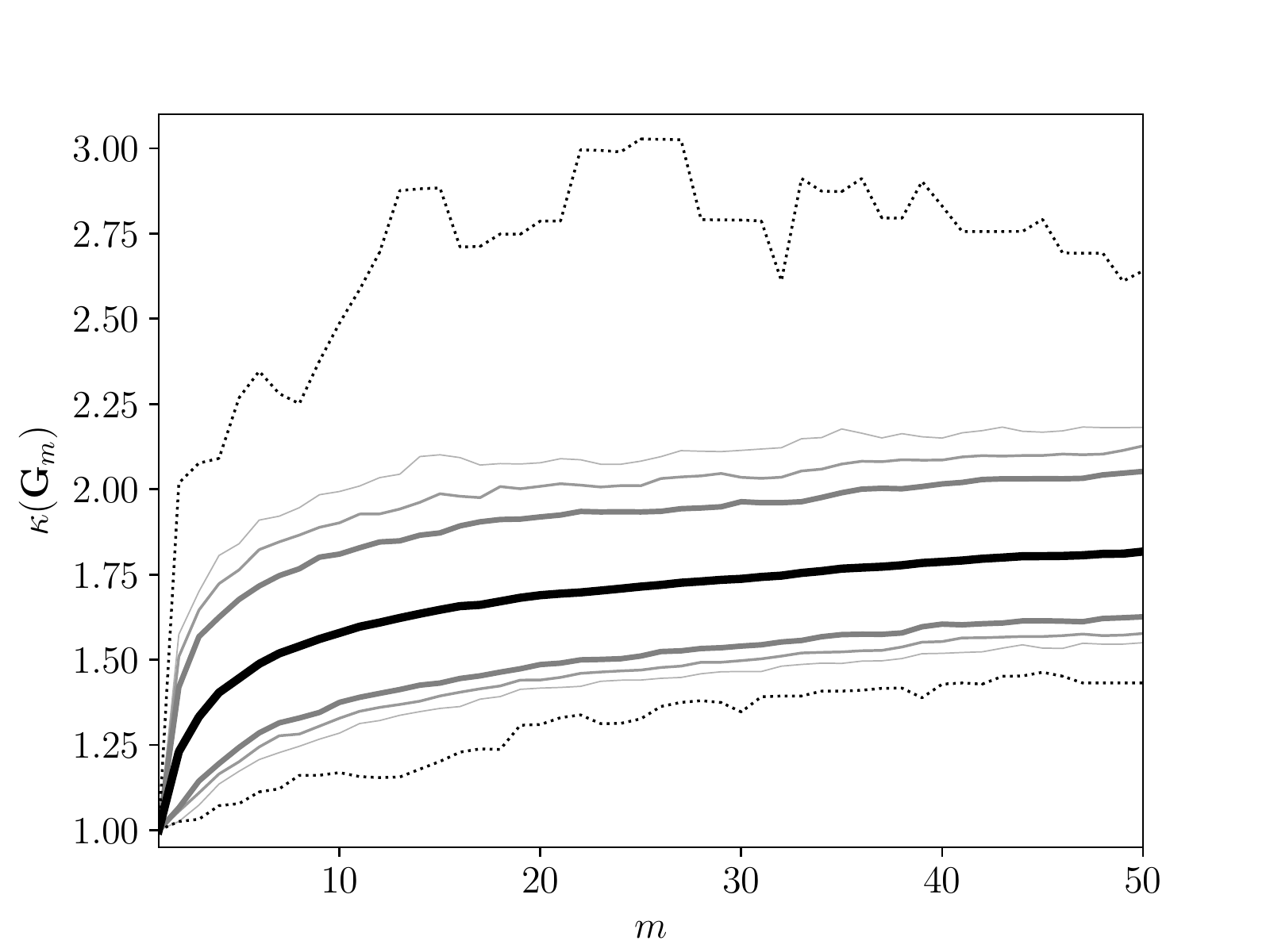} & \includegraphics[width=7.5cm]{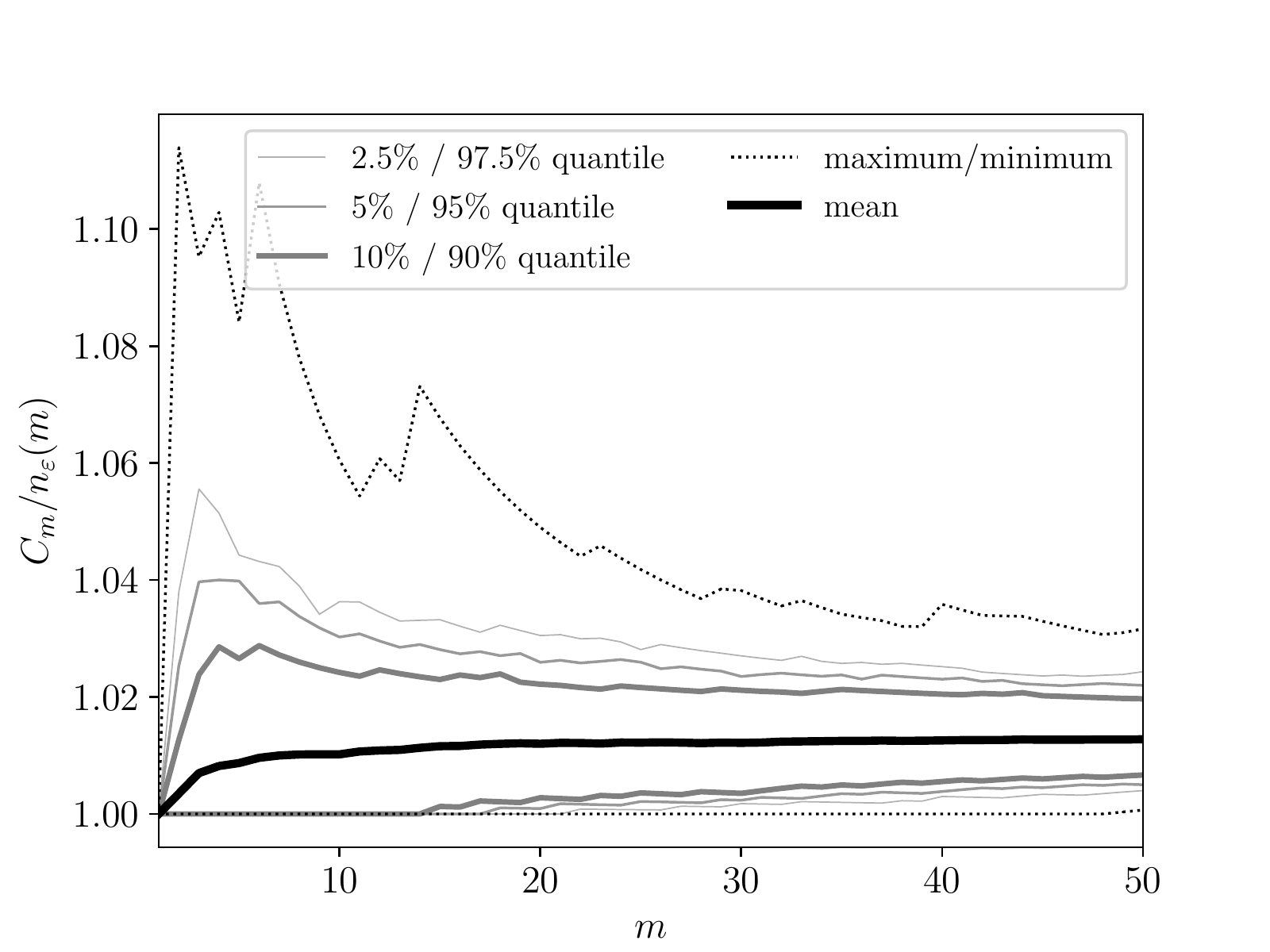} \\
	\small (a)  Gramian condition numbers & \small (b) Cost comparison
	\end{tabular}
	\caption{Results for Algorithm \ref{mixturestep2} with $n(m) = n_\e(m)$ as in \eqref{defstdn}, $\e = 10^{-2}$, applied to subset of Haar basis obtained by random tree refinement.}\label{fig2nhaar}
\end{figure}

A further algorithmic variant consists in applying the inner loop of Algorithm \ref{mixturestep2} until one
is ensured that the stability criterion $\|\bG_m-\bI\|\leq \frac 1 2$ 
is met, so that in particular $\kappa(\bG_m)\leq 3$ holds with certainty.
This procedure is described in Algorithm \ref{mixturestep3}. Note that here the size $\hat n_m$ of the sample $S_m$ is not fixed a priori. 
\begin{algorithm}~
\caption{Sequential sampling with guaranteed condition number}\label{mixturestep3}
\begin{algorithmic}
\vspace{-6pt}
\Require sample $S_m = \{ x_m^1,\ldots, x_m^{\hat n_m}\}$ from $\mu_m$
\Ensure sample $S_{m+1} = \{ x_{m+1}^1,\ldots, x_{m+1}^{\hat n_{m+1}}\}$ from $\mu_{m+1}$
\vspace{3pt}
\hrule
\vspace{6pt}
\State $j:= 1$, $i:= 1$, $\lambda := 1$
\Repeat
\State draw $a_i$ uniformly distributed in $\{ 1, \ldots, m+1 \}$
\If{$a_i = m+1$}
\State draw $x^i_{m+1}$ from $\sigma_{m+1}$
\ElsIf{$j \leq \hat n_m$}
\State  $x^i_{m+1} := x^j_m$
\State $j \gets j+1$ 
\Else
\State draw $x^i_{m+1}$ from $\mu_{m}$ by \eqref{mixturedirect}
\EndIf
\If{$i\geq m+1$}
\State Assemble $\tilde\bG^i_{m+1}$ according to \eqref{gramiandef} using $\{x^1_{m+1},\ldots, x^i_{m+1}\}$
\State $\lambda \gets \|\tilde\bG^i_{m+1}-\bI\|$
\EndIf
\State $i \gets i+1$ 
\Until{$\lambda \leq \frac 1 2$}
\end{algorithmic}
\end{algorithm}

Since the stability criterion is ensured with certainty by this 
third Algorithm, we only need to study the total sampling cost $C_m$. This is illustrated in
Figure \ref{figalg3} in the case (i) of Hermite polynomials.  For a direct comparison to Algorithms \ref{mixturestep} and \ref{mixturestep2}, we compare the costs to $n_\e(m)$ as before, although this value plays no role in Algorithm \ref{mixturestep3}. Figure \ref{figalg3}(a) shows that
with high probability, one has $C_m / n_\e(m) < 1$ for the considered range of $m$ and $\e=10^{-2}$, although this ratio can be seen to increase approximately logarithmically.
A closer inspection shows that $C_m$ tends to behave like $m\log^2 m$, 
as illustrated on Figure \ref{figalg3}(b). This hints that 
an extra logarithmic factor is needed for ensuring stability with certainty for all values of $m$.

\begin{figure}
\begin{tabular}{cc}
	\includegraphics[width=7.5cm]{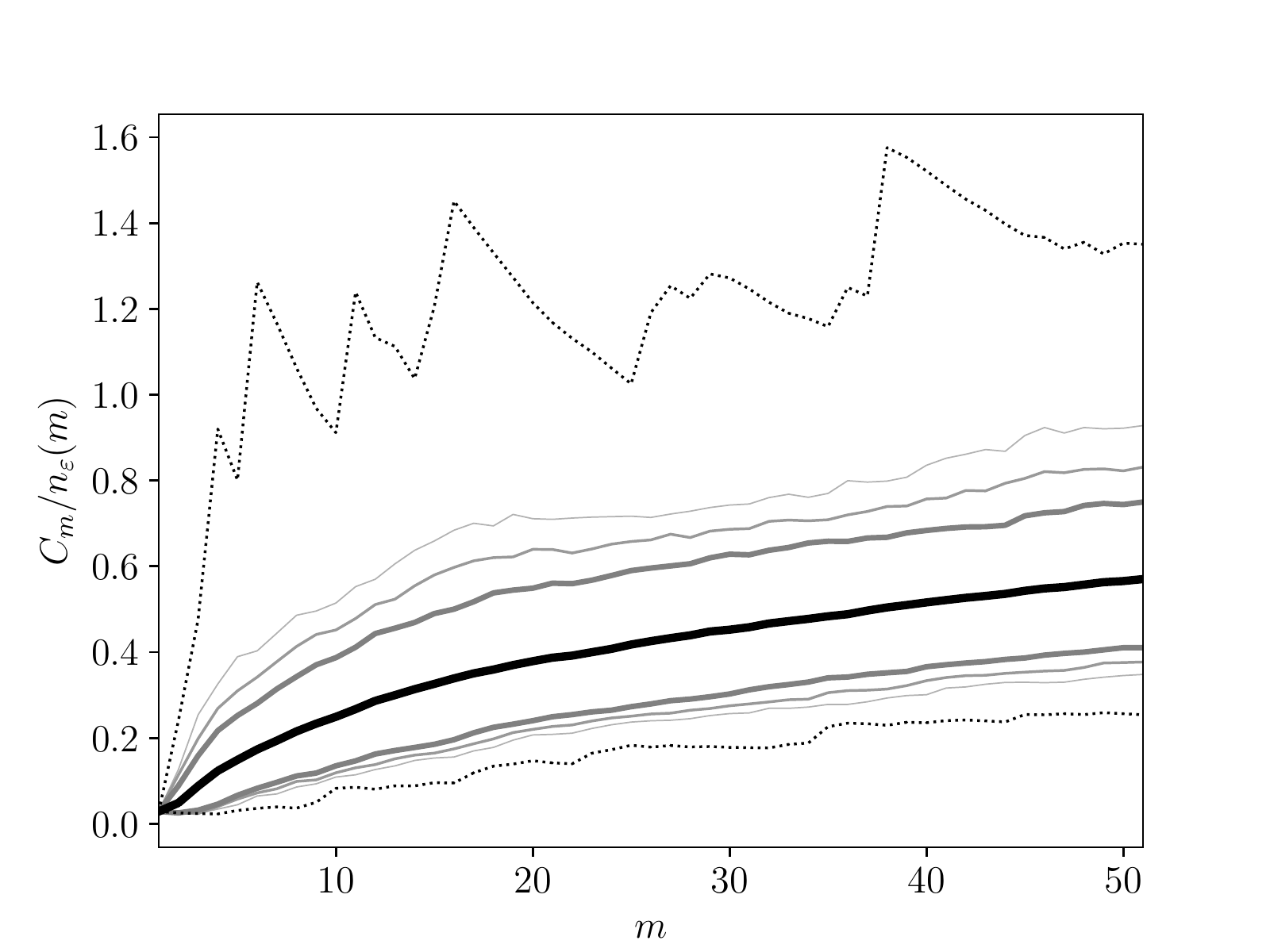} & \includegraphics[width=7.5cm]{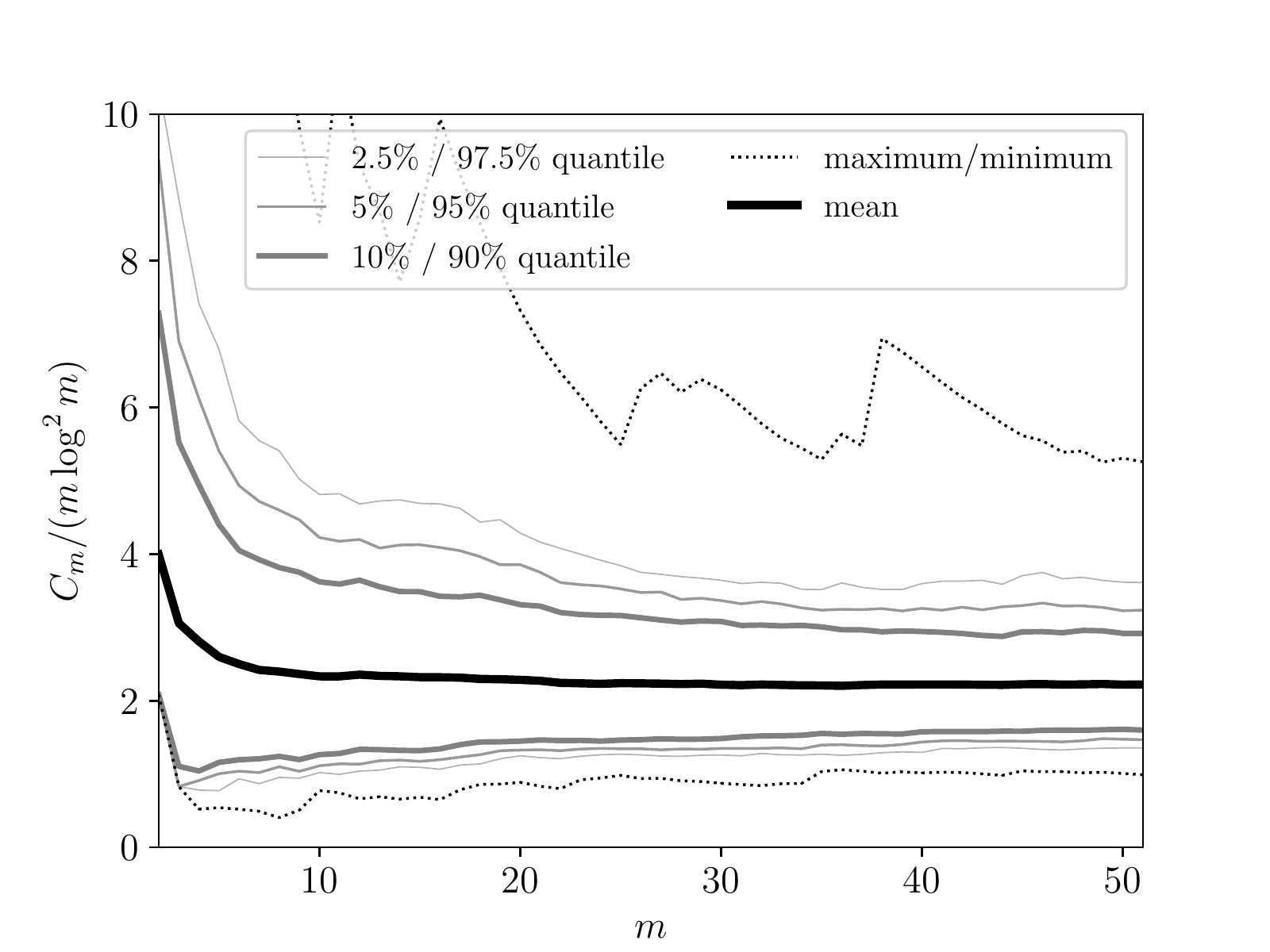} \\
	\small (a)  Ratio between $C_m$ and $n_\e(m)$ & \small (b) Ratio between $C_m$ and $m \log^2 m$.
	\end{tabular}
	\caption{Results for Algorithm \ref{mixturestep3} applied to Hermite polynomials of degrees $0,\ldots,m-1$.}\label{figalg3}
\end{figure}


\begin{thebibliography}{77}
%
\bibitem{AV} D.\ Angluin and L. G. Valiant, \emph{Fast Probabilistic algorithms for Hamiltonian circuits and mappings}, Journal of Computer and System Sciences \textbf{18}, 155-193, 1979.
%
 \bibitem{CCMNT} A. Chkifa, A. Cohen, G. Migliorati, F. Nobile, and R. Tempone,
{\it Discrete least squares polynomial approximation with random evaluations - application
to parametric and stochastic PDEs}, M2AN , \textbf{49}, 815-837, 2015.

\bibitem{Co} A. Cohen {\it Numerical analysis of wavelet methods}, 
Studies in mathematics and its applications, Elsevier, Amsterdam, 2003

\bibitem{CD} A. Cohen and R. DeVore, {\it Approximation of high-dimensional PDEs}, 
Acta Numerica, \textbf{24}, 1-159, 2015.

\bibitem{CDS}  A. Cohen, R. DeVore, and C. Schwab, 
{\it Analytic regularity and polynomial approximation of parametric and stochastic PDEs}, 
Analysis and Applications, \textbf{9}, 11-47, 2011.

\bibitem{CM1} A. Cohen and G. Migliorati, 
\emph{Optimal weighted least squares methods,} SMAI Journal of Computational Mathematics \textbf{3}, 181--203, 2017.

\bibitem{CM2} A. Cohen and G. Migliorati, {\it Multivariate approximation in downward closed polynomial spaces},
to appear in Ian Sloan's 80 Festschrift, 2018.

\bibitem{Da} W. Dahmen {\it Wavelet and multiscale methods for operator equations}, Acta Numer., 55-228, 1997.

\bibitem{De} R. DeVore {\it Nonlinear approximation}, Acta Numer., 51--150, 1998.

\bibitem{Do1} A. Doostan and J. Hampton, {\it Coherence motivated sampling and convergence analysis of least squares polynomial Chaos regression}, Computer Methods in Applied Mechanics and Engineering 290, 73-97, 2015.

\bibitem{Do2} A. Doostan and M. Hadigol, {\it Least squares polynomial chaos expansion: 
A review of sampling strategies}, Computer Methods in Applied Mechanics and Engineering, \textbf{332}, 382-407, 2018.

\bibitem{Do3} A. Doostan and J. Hampton, {\it Basis Adaptive Sample Efficient Polynomial Chaos (BASE-PC)},
preprint, arXiv:1702.01185, 2017

\bibitem{EMN} T. Erd\'elyi, A.P. Magnus and P. Nevai, {\it Generalized Jacobi weights, Christoffel functions, and Jacobi polynomials},  SIAM Journal on Mathematical Analysis, \textbf{25}, 602-614, 1994.

\bibitem{HNTW} A.L. Haji-Ali, F. Nobile, R. Tempone and S. W\"olfer {\it Multilevel weighted least squares polynomial approximation}, preprint, arXiv:1707.00026, 2017

\bibitem{HR} T. Hagerup and C.\ R\"ub, {\it A guided tour of Chernoff bounds},
Information Processing Letters \textbf{33}, 305--308, 1990.

\bibitem{JNZ} J.D. Jakeman, A. Narayan, and T. Zhou, {\it A Christoffel function weighted least squares algorithm for collocation approximations}, Math. Comp., \textbf{86}, 1913-1947, 2017.

\bibitem{MT} G. Mastroianni and V. Totik, {\it Weighted polynomial inequalities with doubling and $a_\infty$ weights},
Constructive approximation \textbf{16}, 37-71, 2000.

\bibitem{Tr} J. Tropp, {\it User-Friendly tail bounds for sums of random matrices},
Found. Comput. Math. \textbf{12}, 389-434, 2012.

\end{thebibliography}
\end{document}